\documentclass[12pt]{amsart}
\date{30 June 2020}
\usepackage{latexsym,amsmath,amsfonts,amscd,amssymb}
\usepackage{stmaryrd}
\usepackage{lscape}
\usepackage{array}   
\usepackage[mathscr]{eucal} 
\usepackage{mathrsfs}
\usepackage{graphicx}
\usepackage{xypic}
\usepackage{calligra}
\usepackage[T1]{fontenc}
\DeclareFontFamily{OT1}{pzc}{}
\DeclareFontShape{OT1}{pzc}{m}{it}{<-> s * [1.10] pzcmi7t}{}
\DeclareMathAlphabet{\mathpzc}{OT1}{pzc}{m}{it}

\DeclareFontFamily{OT1}{rsfs}{}
 \DeclareFontShape{OT1}{rsfs}{n}{it}{<->rsfs10}{}
 \DeclareMathAlphabet{\curly}{OT1}{rsfs}{n}{it}

\theoremstyle{plain}  
\newtheorem{theorem}{Theorem}[section]

\newtheorem*{theorem*}{Theorem}

\newtheorem{proposition}[theorem]{Proposition}

\newtheorem{definition}[theorem]{Definition}
\theoremstyle{remark}

\newtheorem{remark}[theorem]{Remark}

\newtheorem*{claim*}{Claim}
\numberwithin{equation}{section}

\renewcommand{\leq}{\leqslant}

\renewcommand{\geq}{\geqslant}

\renewcommand{\setminus}{\smallsetminus}

\newcommand{\R}{\mathbb{R}}

\newcommand{\Z}{\mathbb{Z}}
\newcommand{\C}{\mathbb{C}}

\newcommand{\AAA}{\curly{A}}
\newcommand{\CCC}{\curly{C}}

\newcommand{\LLL}{\curly{L}}
\newcommand{\PPP}{\curly{P}}
\newcommand{\QQQ}{\curly{Q}}
\newcommand{\SSS}{\curly{S}}

\newcommand{\VV}{\mathbb{V}}

\newcommand{\cM}{\mathcal{M}}

\newcommand{\calR}{\mathcal{R}}

\newcommand{\dbar}{\bar{\partial}}

\newcommand{\SU}{\mathrm{SU}}

\newcommand{\GL}{\mathrm{GL}}
\newcommand{\SL}{\mathrm{SL}}

\DeclareMathOperator{\ad}{ad}
\DeclareMathOperator{\Ad}{Ad}

\DeclareMathOperator{\Hom}{Hom}
\DeclareMathOperator{\End}{End}

\DeclareMathOperator{\Id}{Id}

\DeclareMathOperator{\Aut}{Aut}
\DeclareMathOperator{\aut}{aut}

\DeclareMathOperator{\Out}{Out}
\DeclareMathOperator{\Diff}{Diff}
\DeclareMathOperator{\Mod}{Mod}
\DeclareMathOperator{\Conj}{Conj}

\renewcommand{\phi}{\varphi}

\newcommand{\liem}{\mathfrak{m}}
\newcommand{\liep}{\mathfrak{p}}

\newcommand{\liemc}{\mathfrak{m}^{\mathbb{C}}}
\newcommand{\lieh}{\mathfrak{h}}

\newcommand{\lieg}{\mathfrak{g}}
\newcommand{\liel}{\mathfrak{l}}
\newcommand{\liez}{\mathfrak{z}}

\newcommand{\liet}{\mathfrak{t}}

\renewcommand{\phi}{\varphi}

\DeclareMathOperator{\id}{id}

\begin{document}

\title[Action of the mapping class group on character varieties and
  Higgs bundles]{Action of the mapping class group on \\character varieties and
Higgs bundles}

\author[Oscar Garc{\'\i}a-Prada]{Oscar Garc{\'\i}a-Prada}
\address{Instituto de Ciencias Matem\'aticas \\
  CSIC-UAM-UC3M-UCM \\ Nicol\'as Cabrera, 13--15 \\ 28049 Madrid \\ Spain}
\email{oscar.garcia-prada@icmat.es}

\author[Graeme Wilkin]{Graeme Wilkin}
\address{Department of Mathematics\\
James College, Campus West\\
University of York\\
YO10 5DD\\
United Kingdom}
\email{graeme.wilkin@york.ac.uk}

\thanks{
The first author was partially supported by the Spanish MINECO under the ICMAT Severo Ochoa 
grant No.SEV-2015-0554, and grant No. MTM2013-43963-P and by the European
Commission Marie Curie IRSES  MODULI Programme PIRSES-GA-2013-612534.  The second author was partially supported by Singapore Ministry of Education Academic Research Fund
Tier 1 grant number R-146-000-200-112.
}


\begin{abstract}
We consider the action of 
a finite subgroup of the 
mapping class group $\Mod(S)$  of 
an oriented compact 
surface $S$ of genus $g \geq 2$ on the moduli space $\calR(S,G)$ of 
representations of $\pi_1(S)$ in a connected semisimple real Lie group 
$G$. Kerckhoff's solution of the Nielsen realization problem ensures the
existence of an element $J$  in the Teichm\"uller 
space of $S$ for which $\Gamma$ can be realised as a subgroup of the 
group of automorphisms of $X=(S,J)$ which are holomorphic or antiholomorphic.
We identify the fixed points of the action of $\Gamma$ on $\calR(S,G)$
in terms of $G$-Higgs bundles on $X$
equipped with a certain twisted $\Gamma$-equivariant structure,
where the twisting involves abelian and non-abelian group cohomology
simultaneously. These, in turn, correspond to certain representations of the 
orbifold fundamental group.
When the kernel of the isotropy representation of the maximal compact subgroup
of $G$ is trivial, the fixed points can be described in terms of familiar objects
on  $Y=X/\Gamma^+$, where  $\Gamma^+ \subset \Gamma$ is the maximal subgroup of $\Gamma$
consisting of holomorphic automorphisms of $X$. If $\Gamma=\Gamma^+$
one obtains  actual $\Gamma$-equivariant
$G$-Higgs bundles on $X$, which in turn correspond with 
parabolic Higgs bundles on $Y=X/\Gamma$ (this generalizes work of
Nasatyr \& Steer for $G=\SL(2,\R)$ and Boden, Andersen \& Grove and 
Furuta \& Steer for $G=\SU(n)$). If on the other hand  $\Gamma$ has antiholomorphic
automorphisms, the objects on  $Y=X/\Gamma^+$ correspond with pseudoreal parabolic Higgs
bundles. This is a generalization  in the parabolic setup of the  pseudoreal Higgs bundles 
studied by the first author in collaboration with Biswas  \& Hurtubise.

\end{abstract}

\maketitle

\section{Introduction}
Let $S$ be a compact oriented surface of genus greater than one, and $G$ be
a real connected semisimple Lie group. Consider the moduli space of representations or
character variety $\calR(S,G)$ defined as the space of reductive 
representations of the fundamental group of $S$ in $G$ modulo conjugation 
by elements of $G$.  These are very important varieties that play a central 
role in geometry, topology, higher Teichm\"uller theory and theoretical physics
(see \cite{garcia-prada} for a survey).
A fundamental problem is that of understanding the action of the 
mapping class group or modular group of the surface $\Mod(S)$ in 
$\calR(S,G)$. In this paper, we consider the action of a finite subgroup
$\Gamma$ of $\Mod(S)$ and give a description of the fixed-point 
subvariety.  

A crucial step in our study is provided by a theorem of Kerckhoff
solving the  Nielsen realization problem 
\cite{kerckhoff}. This theorem proves the existence of a complex structure
$J$ on $S$, such that, if $X:=(S,J)$ is the corresponding Riemann surface,
$\Gamma$ is a subgroup of the group of  automorphisms of $X$ which are 
holomorphic or antiholomorphic. We can then use holomorphic methods, and 
in particular the theory of $G$-Higgs bundles over $X$. To define a 
$G$-Higgs bundle, we consider  a maximal compact  subgroup  $H\subset G$, 
and a Cartan decomposition $\lieg=\lieh\oplus \liem$. 
A $G$-Higgs bundle is a pair $(E,\varphi)$ consisting of a $H^\C$-bundle $E$,
where $H^\C$ is the complexification of $H$, and a holomorphic section
$\varphi$ of $E(\liem^\C)\otimes K$, where $E(\liem^\C)$ is the bundle 
associated to the complexification of the isotropy representation of $H$ in 
$\liem$, and
$K$ is the canonical line bundle of $X$.
  The non-abelian Hodge  correspondence establishes a
homeomorphism between $\calR(S,G)$ and the moduli space of polystable
$G$-bundles over $X=(S,J)$ for any complex structure $J$ on $S$. Now, if 
$J$ is the complex structure given by  Kerckhoff's theorem, using the 
 non-abelian Hodge  correspondence one can show that the action 
of an element of $\gamma\in \Gamma$ on $\calR(S,G)$ coincides with the 
natural action of  $\gamma$ on $\cM(X,G)$ via pull-back, if $\gamma$ is 
holomorphic, or the combination
of this with the conjugation defined by the reduction of the $H^\C$-bundle 
to $H$  defined by the solution to
the Hitchin equations, if $\gamma$ is antiholomorphic.
Our problem becomes
then that of analysing the fixed points $\cM(X,G)^\Gamma$ for this action.

The fixed-point subvariety  $\cM(X,G)^\Gamma$ is described in terms of 
$G$-Higgs bundles equipped with a certain twisted $\Gamma$-equivariant 
structure, where the twisting involves a compact conjugation $\tau$ of $H^\C$
and  a group $2$-cocycle $c\in Z^2_\tau(\Gamma,Z')$, where $Z'$ is a 
$\tau$-invariant subgroup of 
the centre of  $H^\C$ and $\gamma\in \Gamma$ acts on $z\in Z'$ trivially if
$\gamma$ is holomorphic and by $\tau(z)$ if $\gamma$ is antiholomorphic. 
We refer to this as a $(\Gamma,\tau,c)$-equivariant structure.
These twisted $\Gamma$-equivariant structures generalise at the same time 
genuine $\Gamma$-equivariant structures when $\Gamma$ consists entirely of 
holomorphic automorphisms of $X$, as well as twisted real structures (referred also
as pseudoreal structures in the literature) when 
$\Gamma$ is the group generated by an antiholomorphic involution of $X$ 
(see \cite{biswas-garcia-prada,biswas-garcia-prada-hurtubise,
biswas-garcia-prada-hurtubise2}). 
When $Z'$ is contained in the kernel of 
the isotropy 
representation 
and  $\Gamma$ is a subgroup of the group of holomorphic automorphisms of $X$, 
these are lifts of true $\Gamma$-equivariant structures on the associated
$G/Z'$-Higgs bundles. 

Assuming that $\Gamma$ is not a group generated by an antiholomorphic involution
of $X$ (as mentioned above, this case is treated in 
\cite{biswas-garcia-prada,biswas-garcia-prada-hurtubise,
biswas-garcia-prada-hurtubise2}), 
it is well-known that there is only a finite number of
points $x\in X$ for which the isotropy subgroups 
$\Gamma_x\subset \Gamma$ for the action of $\Gamma$ on $X$ are different
from the trivial subgroup $\{1\}$, and $\Gamma_x^+$, the subgroup of
$\Gamma_x$ consisting of holomorphic automorphisms, is a cyclic
group. At such points, a $(\Gamma,\tau,c)$-equivariant structure
defines an element $\sigma_x$ in the $c_x$-twisted character variety
of $\Gamma_x$ in $H^\C$, where $c_x\in Z^2(\Gamma_x, Z')$
is the restriction of $c$ to $\Gamma_x$. Here 
$\gamma\in\Gamma_x$ acts trivially on $H^\C$ if $\gamma$ is holomorphic and
by $\tau$ if $\gamma$ is antiholomorphic.
Fixing the cocycle $c$ and the elements $\sigma_x$ 
at the  points with  $\Gamma_x\neq \{1\}$, we define the moduli space of
$(\Gamma,\tau,c)$-equivariant $G$-Higgs bundles  with fixed $\sigma_x$.
Our main result is Theorem \ref{fixed-points-theorem}, where we  show 
that the moduli spaces of  $(\Gamma,\tau,c)$-equivariant $G$-Higgs bundles are
in the fixed-point locus $\cM(X,G)^\Gamma$, and more over, a smooth point in  
$\cM(X,G)^\Gamma$ corresponds to a  point in a moduli space of 
$(\Gamma,\tau,c)$-equivariant $G$-Higgs bundles for some $2$-cocycle $c$. In fact it is only the cohomology class of $c$ which is relevant in the parametrization of fixed points. Using  Theorem \ref{fixed-points-theorem} and a twisted equivariant version of the non-abelian Hodge correspondence (Theorem \ref{equivariant-nahc}), we describe  in
Theorem \ref{fixed-points-theorem-rep} the fix-point locus $\calR(S,G)^\Gamma$ in terms of representations of the orbifold fundamental group for the action of 
$\Gamma$ on $S$.

When $\Gamma$ consists entirely of holomorphic automorphisms of $X$, 
generalising a 
well-known result for vector bundles 
\cite{mehta-seshadri,furuta-steer,nasatyr-steer,biswas,andersen-masbaum,
andersen-grove}, and principal bundles \cite{teleman-woodward,balaji-seshadri},
we establish in Theorem \ref{equivariant-parabolic} a correspondence 
between $\Gamma$-equivariant (that is, without twisting) $G$-Higgs 
bundles over $X$ and parabolic $G$-Higgs bundles over $Y:=X/\Gamma$. The 
weights of the parabolic structure are determined by  the elements 
$\sigma_x$ defined by the equivariant structure, which in this case are simply
elements in the character variety $\Hom(\Gamma_x,H^\C)/H^\C$ of $\Gamma_x$.
In particular, if 
$Z'$ is contained in the kernel of the isotropy representation there is a
map from the moduli space of $G$-Higgs bundles over $X$ to the moduli space
of $G/Z'$-Higgs bundles and hence a map from the moduli space
of $(\Gamma,c)$-equivariant $G$-Higgs bundles over $X$ 
(here there is no twisting by $\tau$) to the 
moduli space of parabolic $G/Z'$-Higgs bundles over $Y$. 
In this situation, using the non-abelian 
Hodge correspondence between parabolic $G$-Higgs bundles  and 
representations of the fundamental group of a punctured surface, proved 
in \cite{biquard-garcia-prada-mundet},  we  relate 
in Theorem \ref{equiv-rep-punctures}
the representations of 
the orbifold fundamental group for the action of 
$\Gamma$ on $S$ to the representations of the fundamental group of 
$S/ \Gamma$ with punctures at the points corresponding to the elements of $S$ with non-trivial isotropy subgroup.

If we allow $\Gamma$ to contain antiholomorphic automorphisms, 
and $\Gamma^+$ is 
the subgroup of $\Gamma$ consisting of holomorphic automorphisms, we consider
the Riemann surface $Y:=X/\Gamma^+$. On this surface there is a residual 
antiholomorphic action of $\Z/2\cong \Gamma/\Gamma^+$. Now, if the restriction 
of $c$ to $\Gamma^+$ is trivial,  $(\Gamma,\tau,c)$-equivariant 
$G$-Higgs bundles on $X$ are in correspondence with a pseudoreal  parabolic 
$G$-Higgs bundles on $Y$ as described in \cite{calvo-garcia-prada-perez}. 
This is a generalization in the parabolic set-up of the notion of  pseudoreal Higgs bundle
studied in \cite{biswas-garcia-prada-hurtubise,biswas-garcia-prada,biswas-calvo-garcia-prada}. 
Again, using the non-abelian Hodge correspondence in \cite{biquard-garcia-prada-mundet}, we  relate 
in Theorem \ref{equiv-rep-punctures-pseudo}
the representations of 
the orbifold fundamental group for the action of 
$\Gamma$ on $S$ to the representations of the $\Z/2$-orbifold fundamental 
group of 
$S/\Gamma^+$ with punctures at the points corresponding to the elements of $S$ with non-trivial isotropy subgroup.

The more general $(\Gamma,\tau,c)$-equivariant objects on $X$, correspond to twisted 
parabolic objects  on $Y:=X/\Gamma^+$ in a more involved way, and  will be treated in
a separate paper.

In the process of writing up this paper, we came across 
the recent related work \cite{schaffhauser,wu,heller-schaposnik}. 

\noindent{\bf Acknowledgements}. We wish to thank Steve Kerckhoff, Jochen
Heinloth and Peter Gothen for useful discussions, and   NUS (Singapore),  
CMI (Chennai), 
Bernoulli Center (Lausanne) and ICMAT (Madrid) for hospitality and support. 
We also want to thank the referee for the comments and suggestions.
\section{Moduli space of representations and the mapping class group}
In this section $S$ is an oriented smooth compact surface of genus $g\geq 2$.

\subsection{Moduli space of representations}

Let $G$ be  a connected real 
reductive Lie group. By a {\bf representation} of $\pi_1(S)$ in
$G$ we mean a homomorphism $\rho\colon \pi_1(S) \to G$.
The set of all such homomorphisms, denoted
$\Hom(\pi_1(S),G)$,  is an analytic  variety, which is algebraic
if $G$ is algebraic.
The group $G$ acts on $\Hom(\pi_1(S),G)$ by conjugation:
$$
(g \cdot \rho)(\gamma) = g \rho(\gamma) g^{-1}
$$
for $g \in G$, $\rho \in \Hom(\pi_1(S),G)$ and
$\gamma\in \pi_1(S)$. If we restrict the action to the subspace
$\Hom^+(\pi_1(S),g)$ consisting of reductive representations,
the orbit space is Hausdorff.  By a {\bf reductive representation} we mean
one that, composed with the adjoint representation in the Lie algebra
of $G$, decomposes as a sum of irreducible representations.
If $G$ is algebraic this is equivalent to the Zariski closure of the
image of $\pi_1(S)$ in $G$ being a reductive group.
(When $G$ is compact every representation is reductive).  The
{\bf moduli space of representations} or {\bf character variety}
of $\pi_1(S)$ in $G$
is defined to be the orbit space
$$
\calR(S,G) = \Hom^{+}(\pi_1(S),G) / G. 
$$

It has the structure of an  analytic variety (see e.g.~\cite{goldman}) 
which is algebraic if $G$ is algebraic (see e.g.~ \cite{richardson}) and is real if $G$ is real or  
complex  if $G$ is complex. If $G$ is complex then $\calR(S,G)$ can also be expressed as the GIT quotient 
$$
\calR(S,G) = \Hom(\pi_1(S),G)\sslash  G .
$$

Let  $\rho:\pi_1(S)\to G$ be a representation of $\pi_1(S)$ in
$G$. Let $Z_G(\rho)$ be the centralizer in $G$ of
$\rho(\pi_1(S))$. We say that  $\rho$ is {\bf  irreducible} if and
only if it is reductive and $Z_G(\rho)=Z(G)$, where $Z(G)$ is the
centre  of $G$. 

\subsection{The mapping class group}

The {\bf mapping class group} or {\bf modular group} of $S$ is defined
as 
$$
\Mod(S)=\pi_0\Diff(S),
$$
where $\Diff(S)$ is the group of diffeomorphisms of $S$. 
We also consider the subgroup
$$
\Mod^+(S)=\pi_0\Diff^+(S),
$$
where 
$\Diff^+(S)$ is the subgroup of $\Diff(S)$ consisting of orientation-preserving
diffeomorphisms. We have an exact sequence

\begin{equation}\label{modular-extension}
1\to \Mod^+(S)\to \Mod(S) \to \Z/2\to 1.
\end{equation}

By the Dehn--Nielsen--Baer theorem, $\Mod(S)$ is isomorphic to 
$\Out(\pi_1(S))$, the group of  outer automorphisms of $\pi_1(S)$, 
and hence acts in the obvious way on $\calR(S,G)$.

Let $\Gamma\subset \Mod(S)$ be a finite subgroup. 
The main goal 
of this paper is to investigate the fixed points $\calR(S,G)^\Gamma$. 
A crucial step to do this
is provided by Kerckhoff's solution of the {\bf Nielsen realization problem} 
\cite{kerckhoff}. To explain this, let $J$ be an element in the Teichm\"uller space of $S$ and
$X=(S,J)$ be the corresponding Riemann surface. Denote by $\Aut(X)$ the group consisting of automorphisms
of $S$ which are holomorphic or antiholomorphic with respect to $J$. If  $\Aut^+(X)$  is the subgroup of $\Aut(X)$ 
consisting of holomorphic automorphisms of $X$, there is an exact sequence
\begin{equation}\label{modular-extension}
1\to \Aut^+(X)\to \Aut(X) \to \Z/2\to 1.
\end{equation}

\begin{theorem}\label{kerckhoff}
Let $\Gamma\subset \Mod(S)$ be a finite subgroup.
There exists an element $J$ in the Teichm\"uller space of $S$ such that $\Gamma\subset \Aut(X)$, where 
$X=(S,J)$.  In particular,
if $\Gamma\subset  \Mod^+(S)$, one has $\Gamma\subset \Aut^+(X)$. Moreover,
if $X$ is not hyperelliptic, $\Gamma=\Aut^+(X)$
if $\Gamma\subset \Mod^+(S)$, and $\Gamma=\Aut(X)$
if $\Gamma$ is not contained in $\Mod^+(S)$ and  $\Gamma^+\neq \{1\}$.
\end{theorem}

\begin{remark} This had been proved by Nielsen \cite{nielsen} for 
cyclic groups and
by Fenchel \cite{fenchel} for solvable groups.
Thanks to Theorem \ref{kerckhoff} the problem of studying the 
action of $\Gamma$ on $\calR(S,G)$ can be reduced to  studying
the action of $\Gamma$ on the moduli space of $G$-Higgs bundles on $X$.
\end{remark}

\subsection{Moduli space of $G$-Higgs bundles}\label{higgs-bundles}

Here $X$ is a compact  Riemann surface
and  $G$ is a connected real reductive Lie group. We fix a maximal compact 
subgroup $H$ of $G$. 
The Lie algebra $\lieg$ of $G$ is equipped with an involution $\theta$ 
that gives the Cartan decomposition  $\lieg=\lieh + \liem$,
where $\lieh$ is the Lie algebra of $H$. We fix a  metric  $B$ in $\lieg$
with respect to  which the Cartan decomposition is orthogonal.
This metric  is positive definite on $\liem$ and negative definite on $\lieh$.  
We have $[\liem,\liem]\subset \lieh$, $[\liem,\lieh]\subset\liem$.
From the isotropy representation $H\to \Aut(\liem)$, we obtain the
representation $\iota: H^\C \to \Aut(\liem^\C)$.
When $G$ is semisimple we take $B$ to be the Killing form.
In this case $B$ and a choice of a maximal compact subgroup $H$ 
determine a Cartan decomposition (see \cite{knapp} for details).

{\bf A $G$-Higgs bundle} on  $X$ consists of a holomorphic principal
$H^\C$-bundle  $E$ together with a holomorphic section $\varphi\in
H^0(X,E(\liem^\C)\otimes K)$,
where  $E(\liem^\C)$ is the associated vector bundle with fibre 
$\liem^\C$ via the complexified isotropy representation, and $K$ is 
the canonical line bundle of $X$.

If $G$ is compact,  $H=G$ and $\liem=0$. A $G$-Higgs bundle is hence 
simply  a holomorphic principal  $G^\C$-bundle.
If  $G=H^\C$, where now $H$ is a compact Lie group,  $H$ is 
a maximal compact subgroup of $G$, and $\liem=i\lieh$. In this case, a 
$G$-Higgs 
bundle is a principal $H^\C$-bundle together with a section 
$\varphi\in H^0(X,E(\lieh^\C)\otimes K)=H^0(X, E(\lieg) \otimes K)$, where 
$E(\lieg)$ is the adjoint bundle. This is the original definition for 
complex Lie groups given by Hitchin in \cite{hitchin:duke}.

There is a  notion of stability for $G$-Higgs
bundles (see \cite{garcia-prada-gothen-mundet}). To explain 
this  we  consider the   parabolic subgroups of $H^\C$ 
defined for $s\in i\lieh$ as
\begin{equation}\label{parabolic}
P_s =\{g\in H^\C\ :\ e^{ts}ge^{-ts}\text{ is bounded as }t\to\infty\}.
\end{equation}
A Levi subgroup of $P_s$ is given by $L_s =\{g\in H^\C\ :\ \Ad(g)(s)=s\},$.
Their Lie algebras are given by
\begin{align*}
\mathfrak{p}_s &=\{Y\in\lieh^\C\ :\ \Ad(e^{ts})Y\text{ is
bounded as }t\to\infty\},\\
\mathfrak{l}_s &=\{Y\in\lieh^\C\ :\ \ad(Y)(s)=[Y,s]=0\}.
\end{align*}

We  consider the subspaces
\begin{align*}
&\liem_s=\{Y\in \liem^\C\ :\ \iota(e^{ts})Y
\text{ is bounded as}\;\; t\to\infty\}\\
&\liem^0_{s}=\{Y\in \liem^\C\ :\ \iota(e^{ts})Y=Y\;\;
\mbox{for every} \;\; t\}.
\end{align*}

One has that $\liem_s$ is invariant under the action of $P_{s}$ and $\liem^0_{s}$
is invariant under the action of $L_{s}$. 

An  element $s\in i\lieh$ defines a character $\chi_s$ of $\liep_s$ since
$\langle s,[\liep_s,\liep_s]\rangle=0$.  Conversely, by the isomorphism 
$\left( \liep_s/ [\liep_s,\liep_s]\right)^* \cong \liez_{L_s}^*$, 
where $\liez_{L_s}$ is the centre of the Levi subalgebra $\liel_s$, a 
character $\chi$ of $\liep_s$ is given by an element in $\liez_{L_s}^*$, 
which gives, via the invariant metric, an element of 
$s_\chi\in \liez_{L_s}\subset i\lieh$. When $\liep_s\subset \liep_{s_\chi}$, 
we say that $\chi$ is an antidominant character of $\liep$. 
When  $\liep_s=\liep_{s_\chi}$ we say that 
$\chi$ is a strictly antidominant character. Note that for $s\in i\lieh$, 
$\chi_s$ is a strictly antidominant character of $\liep_s$.

Let now $(E,\varphi)$ be a $G$-Higgs bundle over $X$, 
 and let $s\in i\lieh$. Let $P_s$ be defined as above.
For $\sigma\in \Gamma(E(H^\C/P_s))$ a reduction of the structure group 
of $E$ from $H^\C$ to $P_s$, we define the degree relative to $\sigma$ and
$s$, or equivalently to $\sigma$ and $\chi_s$ in terms of the 
curvature of connections using Chern--Weil theory.
 For this, define $H_s=H\cap L_s$ and
$\lieh_s=\lieh\cap\liel_s$. 
Then $H_s$ is a maximal compact subgroup of $L_s$, so the inclusions
$H_s\subset L_s$ is a homotopy equivalence. Since the inclusion
$L_s\subset P_s$ is also a homotopy equivalence, given a reduction
$\sigma$ of the structure group of $E$ to $P_s$ one can
further restrict the structure group of $E$ to $H_s$ in a unique
way up to homotopy. Denote by $E'_{\sigma}$ the resulting $H_s$
principal bundle.
Consider now a connection $A$
on $E'_{\sigma}$ and let
$F_A\in\Omega^2(X,E'_{\sigma}(\lieh_s)$ be  its
curvature. Then $\chi_s(F_A)$ is a $2$-form on $X$ with
values in $i\R$, and 
\begin{equation}\label{degree-chern-weil}
\deg(E)(\sigma,s):=\frac{i}{2\pi}\int_X \chi_s(F_A).
\end{equation}

We define the subalgebra $\lieh_{\ad}$  as follows. 
Consider the decomposition $ \lieh = \liez + [\lieh, \lieh] $, where $\liez$ 
is the centre of $\lieh$, and the isotropy representation 
$\ad= \ad:\lieh\to \End(\liem)$. Let $\liez'=\ker(\ad_{|\liez})$  and 
take $\liez''$ such that $\liez=\liez'+\liez''$. Define the subalgebra 
$\lieh_{\ad} := \liez'' + [\lieh, \lieh]$. The subindex $\ad$ denotes that 
we have taken away the part of the centre $\liez$ acting trivially via 
the isotropy representation $\ad$.

\begin{definition}
\label{def:L-twisted-pairs-stability} 
 We say that a $G$-Higgs bundle $(E,\varphi)$ is:

{\bf semistable} if for any $s\in i\lieh$ and any holomorphic reduction 
$\sigma\in\Gamma(E(H^\C/P_s))$ such that $\varphi\in
 H^0(X,E_{\sigma}(\liem_s)\otimes K)$, 
we have that $\deg(E)(\sigma,s)\geq 0$;

 {\bf stable} if for any $s\in i\lieh_{\ad}$ and any holomorphic
  reduction $\sigma\in\Gamma(E(H^\C/P_s))$ such that $\varphi\in
  H^0(X,E_{\sigma}(\liem_s)\otimes K)$, we have that
  $\deg(E)(\sigma,s)> 0$;

{\bf polystable} if it is semistable and for
any $s\in i\lieh_{\ad}$ and any holomorphic reduction 
$\sigma\in\Gamma(E(H^\C/P_s))$ 
such that $\varphi\in  H^0(X,E_{\sigma}(\liem_s)\otimes K)$ and  
$\deg(E)(\sigma,s)=0$, there is a holomorphic reduction 
of the structure group $\sigma_L\in\Gamma(E_{\sigma}(P_s/L_s))$ to a Levi
subgroup $L_s$ such that  $\varphi\in H^0(X,E_{\sigma_L}(\liem_s^0)\otimes K)
\subset H^0(X,E_{\sigma}(\liem_s)\otimes K)$.
\end{definition}

We define the {\bf moduli space of polystable $G$-Higgs bundles} 
$\cM(X,G)$  as the set of isomorphism classes of 
polystable  $G$-Higgs bundles on $X$. A GIT construction of this space has been given
by Schmitt \cite{schmitt:2008}. 

The notion of stability emerges from the study of the
Hitchin equations.  The equivalence between the existence of solutions to 
these equations and the polystability of Higgs bundles is 
given by the following
(see \cite{garcia-prada-gothen-mundet}).
\begin{theorem}\label{theo:hk-twisted-pairs}
Let $(E,\varphi)$ be a $G$-Higgs bundle over a Riemann surface $X$.
 Then $(E,\varphi)$ is polystable if and only if 
there exists a reduction $h$ of the structure group of $E$ from $H^\C$ to
$H$,   such that
\begin{equation}\label{eq:Hitchin-Kobayashi-h}
F_h - [\varphi,\tau_h(\varphi)]=0\\
\end{equation}
where $\tau_h:\Omega^{1,0}(E(\liem^\C))\to \Omega^{0,1}(E(\liem^\C))$ is the
combination of  the anti-holomorphic involution in $E(\liem^\C)$ defined 
by the compact real form at each point determined by $h$  and the 
conjugation of $1$-forms, and $F_h$ is the 
curvature of the unique $H$-connection compatible with the holomorphic 
structure of $E$. 
\end{theorem}

 A $G$-Higgs bundle $(E,\varphi)$ is said to be {\bf simple} if 
 \begin{math}
 \mathrm{Aut}(E,\varphi)=Z(H^{\C})\cap \ker(\iota)
\end{math}
where $Z(H^{\C})$ the centre of $H^\C$.
A $G$-Higgs bundle $(E,\varphi)$ is said to be
{\bf infinitesimally simple} if the infinitesimal automorphism
space $\aut(E,\varphi)$ is isomorphic to
$ H^0(X,E(\ker d\iota \cap Z(\lieh^{\C}))$ where  $Z(\lieh^{\C})$ denotes the Lie algebra of $Z(H^{\C})$.

Thus a $G$-Higgs bundle is (infinitesimally) simple if its
(infinitesimal) automorphism group is
as small as possible.
It is clear that a simple $G$-Higgs bundle is
  infinitesimally simple.  If $G$ is complex then $\iota$ is the
  adjoint representation and $(E,\varphi)$ is simple
  (resp.\ infinitesimally simple) if $\Aut(E,\varphi)=Z(G)$
  (resp.\ $\aut(E,\varphi)=Z(\lieg)$).

The basic link between  representations  of $\pi_1(S)$  and 
Higgs bundles is given by the {\bf non-abelian Hodge correspondence} due to Hitchin, 
Donaldson, Simpson, Corlette and others (see \cite{garcia-prada-gothen-mundet} 
and references there).
\begin{theorem}\label{na-Hodge}
Let $S$ be a compact surface and $X=(S,J)$ be the Riemann surface defined
by any complex structure $J$ on $S$.
Let $G$ be a real connected semisimple Lie group. There is a homeomorphism
$\calR(S,G) \stackrel{\cong}{\longrightarrow} \cM(X,G)$, where the image of 
the irreducible representations is the subspace of stable and simple $G$-Higgs 
bundles.
\end{theorem}

A key step to go from a polystable $G$-Higgs bundle $(E,\varphi)$ over $X$ 
to a representation  $\rho$ of $\pi_1(S)$ in $G$ is given by the relation

\begin{equation}\label{naht-nabla}
\nabla=\dbar_E-\tau_h(\dbar_E)+\varphi -\tau_h(\varphi),
\end{equation}
where $\nabla$ is the flat connection corresponding to $\rho$, $\dbar_E$ is 
the Dolbeault operator of $E$ and $\tau_h$ is provided by the solution to the
Hitchin equations in Theorem \ref{theo:hk-twisted-pairs}. The converse 
construction is provided by the Donaldson--Corlette theorem on the existence of harmonic
metrics on a reductive flat bundle given in  \cite{donaldson,corlette}.

\begin{remark}
Theorem \ref{na-Hodge} can also be extended to (non-connected) reductive groups.
The presence of a continuous centre in $G$ requires replacing the fundamental group of $S$ by its universal central extension.
\end{remark}


From Theorems \ref{kerckhoff} and \ref{na-Hodge} we conclude the following.

\begin{proposition}\label{action-correspondence}
Let $\Gamma\subset \Mod(S)$ be a finite subgroup and 
$\Gamma^+=\Gamma \cap \Mod^+(S)$. Let $J$ be a 
complex structure given by Kerckhoff's theorem and $X=(S,J)$ be the 
corresponding Riemann surface.  Under the non-abelian Hodge 
correspondence $\calR(S,G)\cong \cM(X,G)$ given by Theorem \ref{na-Hodge}, 
the action of $ \Gamma$ on $\calR(S,G)$
coincides with the following  action of $\Gamma$ on  $\cM(X,G)$:

$$
\gamma\cdot(E,\varphi)= \begin{cases}
(\gamma^*E,\gamma^*\varphi) & \mbox{if} \;\; \gamma\in \Gamma^+\\
(\gamma^*\tau_h(E),\gamma^*\tau_h(\varphi)) & \mbox{if} \;\; \gamma\notin \Gamma^+
\end{cases}
$$

\noindent
where $\tau_h$ is  given by Theorem \ref{theo:hk-twisted-pairs},
$\tau_h(E):=E\times_{\tau_h}(H^\C)$ and $\tau_h(\varphi)$ is as in Theorem
\ref{theo:hk-twisted-pairs}.
 We thus have that for this action  $\calR(S,G)^\Gamma$ and 
$\cM(X,G)^\Gamma$ are in bijective correspondence.
\end{proposition}
\begin{proof}
Given any $\gamma \in \Gamma \subset \Mod(S)$, Kerckhoff's theorem 
\cite[Thm. 5]{kerckhoff} 
guarantees a unique diffeomorphism $f$ in the isotopy class of $\gamma$ such 
that $f^* J = J$ if $\gamma\in \Gamma^+$ or  $f^* J = -J$ if $\gamma\notin
\Gamma^+$. The action of $\gamma$ on $\calR(S, G)$ is defined by 
$\gamma \cdot [\rho] = [f^* \rho] = [\rho \circ f_*]$, which induces an action 
on the space of equivalence classes of flat connections given by 
$\gamma \cdot [\nabla] = [f^* \nabla]$ if $\gamma\in \Gamma^+$  or
$\gamma \cdot [\nabla] = [-f^* \nabla]$ if $\gamma\notin \Gamma^+$.
To find the induced action of $\gamma$ on $\cM(X, G)$ via Theorem 
\ref{na-Hodge} (which is well-defined since $f^* J = \pm J$) 
we recall that the flat connection $\nabla$ associated
to a polystable $G$-Higgs bundle $(E,\varphi)$ is given by (\ref{naht-nabla}),
and observe  that $\tau_h(\dbar_E)$  is the Dolbeault operator of $\tau_h(E)$.
Thus proving the statement. 
\end{proof}


\section{Twisted equivariant structures on principal bundles and 
associated vector bundles}

In this section $X$ is a compact Riemann surface of genus bigger than one,
$\Gamma\subset \Aut(X)$,
$G$  is a connected {\bf complex} reductive Lie group,
and $\tau$ is a conjugation of $G$ (not necessarily the compact conjugation).
We will write $\Gamma=\Gamma^+\cup \Gamma^-$, where $\Gamma^+$ is
the subgroup of $\Gamma$ consisting of holomorphic automorphisms and
$\Gamma^-$ is the coset  consisting of antiholomorphic automorphisms.

\subsection{Twisted equivariant structures on a principal bundle}
\label{twist-eq-bun}


Let $Z:=Z(G)$ be the centre of $G$. Consider the action of $\tau$ on $Z$  and 
let $Z'\subset Z$ be a subgroup invariant under this action.
Consider the action of $\Gamma$ on $Z'$ given by 
\begin{equation}\label{action-on-Z}
 z^\gamma= \begin{cases}
z&\mbox{if}\;\;  \gamma\in \Gamma^+\\
\tau(z)&\mbox{if}\;\; \gamma\in \Gamma^-.
\end{cases}
\end{equation}
Let   
$c\in Z^2_\tau(\Gamma,Z')$ be a $2$-{\bf cocycle} for this action.
This is a map $c:\Gamma\times \Gamma\to Z'$ satisfying the 
cocycle condition
$$
c(\gamma',\gamma'')^\gamma c(\gamma,\gamma'\gamma'')=
c(\gamma\gamma',\gamma'')c(\gamma,\gamma').
$$
  
These objects emerge in the study of ``lifts'' to $G$  of {\bf non-abelian
1-cocycles} in $Z^1(\Gamma,G/Z')$ for the action of $\Gamma$ on $G$
given by $ g^\gamma=g$ if $\gamma$ is holomorphic and  
$g^\gamma=\tau(g)$ if $\gamma$ is antiholomorphic. In particular,
if $\Gamma=\Gamma^+$,  the action of $\Gamma$ on $G$
is trivial and $Z^1(\Gamma,G/Z')=\Hom(\Gamma,G/Z')$, that is the 
1-cocycles are simply representations of $\Gamma$ in $G/Z'$. 

Let $E$ be a holomorphic $G$-bundle over $X$.  Let $c\in Z^2_\tau(\Gamma,Z')$.  
A {\bf $(\Gamma,\tau,c)$-equivariant structure on $E$} (or simply {\bf twisted
$\Gamma$-equivariant structure} if there is no need to specify $\tau$ and $c$) consists of a collection
of  maps $\widetilde{\gamma}:E\to E$ covering $\gamma:X\to X$ for 
every $\gamma\in \Gamma$, satisfying 

$$
\widetilde{\gamma}(eg)=\begin{cases} 
\widetilde{\gamma}(e)g  \;\;\mbox{and}\;\; \gamma\;\; \mbox{holomorphic}& \mbox{if}\;\; \gamma\in \Gamma^+\\
\widetilde{\gamma}(e)\tau(g)\;\; \mbox{and}\;\; \gamma\;\; \mbox{antiholomorphic}& \mbox{if} \;\; \gamma\in \Gamma^-,
\end{cases}
$$

$$
\widetilde{\gamma\gamma'}=c(\gamma,\gamma')\widetilde{\gamma}\widetilde{\gamma}',
$$
and $\widetilde{\Id_X}=\Id_E$. This imposes the condition $c(\gamma,1)=1$ for
every $\gamma\in \Gamma$.

When $c$ is the trivial cocycle $1$ we will refer to a 
$(\Gamma,\tau,1)$-equivariant structure as a $(\Gamma,\tau)$-equivariant structure or
a $\tau$-twisted $\Gamma$-equivariant structure. If $\Gamma=\Gamma^+$, we take 
$\tau$ to be the identity and we refer to a $(\Gamma,1,c)$-equivariant structure
as a $(\Gamma,c)$-equivariant structure. If, moreover $c=1$, then we obtain
a genuine $\Gamma$-equivariant structure on $E$.

Let $\Aut(E)$ be the group of holomorphic automorphisms of $E$ covering the 
identity of $X$, and let 
$\Aut_{\Gamma,\tau}(E)$ be the group of bijective maps $f:E\to E$ defined by 
\begin{equation}\label{twisted-autos}
f(eg)=\begin{cases}
f(e)g\;\;\mbox{and}\;\; f\;\; \mbox{holomorphic}&\mbox{if}\;\; f\;\;
\mbox{covers} \;\; \gamma\in \Gamma^+\\
f(e)\tau(g)\;\;\mbox{and}\;\; f\;\; \mbox{antiholomorphic}&\mbox{if}\;\; f\;\;\mbox{covers} \;\; \gamma\in \Gamma^-.
\end{cases}
\end{equation}

There is  an exact
sequence
\begin{equation}\label{exact-aut}
1\to \Aut(E)\to \Aut_{\Gamma,\tau}(E) \to \Gamma.
\end{equation}

A  $(\Gamma,\tau,c)$-equivariant structure on $E$ is simply
a twisted representation $\Gamma \to \Aut_{\Gamma,\tau}(E)$ with cocycle $c$, that is
a map $\sigma:\Gamma \to \Aut_{\Gamma,\tau}(E)$ such that
$$
\sigma(\gamma\gamma')=c(\gamma,\gamma')\sigma(\gamma)\sigma(\gamma').
$$
This is clear since, if  $E'$ is the $G/Z'$-principal bundle 
associated to $E$ via
the projection $G\to G/Z'$, a $(\Gamma,\tau,c)$-equivariant structure on 
$E$ defines a  $(\Gamma,\tau)$-equivariant structure on $E'$, and  we  
have an exact sequence
$$
1\to Z' \to \Aut_{\Gamma,\tau}(E) \to \Aut_{\Gamma,\tau}(E') \to  1.
$$

Two twisted
$\Gamma$-equivariant structures on $E$ for the same $\tau$ and for two 
cocycles $c$ and $c'$ define the same 
$(\Gamma,\tau)$-equivariant structure on $E'$ if and only if 
there exists a   function $f:G\to Z'$ such that the corresponding 
twisted representations
$\sigma$ and $\sigma'$ of $\Gamma$ in $\Aut_{\Gamma,\tau}(E)$ are related by 
$\sigma'=f\sigma$,
and 
\begin{equation}\label{coboundary}
c'(\gamma,\gamma')=f(\gamma\gamma')f(\gamma)^{-1}f(\gamma')^{-1}
c(\gamma,\gamma').
\end{equation}
This defines a natural equivalence relation in the set of 
$(\Gamma,\tau,c)$-equivariant structures on $E$, whose equivalence classes 
are parametrised by the cohomology group $H^2(\Gamma,Z')$.

\begin{remark}
Of course if $Z'=Z$, $G/Z'=\Ad(G)$ and $E'=P(E):=E/Z$.
\end{remark}

There is an alternative way of thinking of a $(\Gamma,\tau,c)$-equivariant 
structure as a $\tau$-twisted  equivariant structure on $E$ for the action of a 
larger  group. Namely, the $2$-cocycle $c$ defines an extension of groups
$$
1\to Z'\to \Gamma_c \to \Gamma \to 1.
$$

Two cocycles are cohomologous if and only if  the corresponding 
extensions are equivalent, i.e. equivalence classes of extensions of $\Gamma$ 
by  $Z'$ with the action of $\Gamma$ on $Z'$ given by \ref{action-on-Z} 
are parametrised  by  $H^2_\tau(\Gamma,Z')$. 

We have the following.

\begin{proposition}
$(\Gamma,\tau,c)$-equivariant structures on $E$ are in bijection with 
central $(\Gamma_ c,\tau)$-equivariant structures on $E$, where $\Gamma_c$ 
acts on $X$ and on $Z'$ via the projection  $\Gamma_c \to \Gamma$, and by central we mean that
the action of $Z'$ in the kernel of the extension above is the natural action 
of $Z'$ on $E$.
\end{proposition}
\begin{proof}

It follows from  group representation theory (see \cite{rotman} for example)
that 
a twisted representation $\Gamma\to \Aut_{\Gamma,\tau}(E)$ with cocycle $c$ 
is equivalent to a representation  $\rho:\Gamma_c \to  \Aut_{\Gamma_c,\tau}(E)$  
fitting in the  following commutative diagram, where $\widetilde{\rho}$
is the induced representation

\begin{displaymath}
  \begin{CD}
  0@>>>Z' @>>>\Gamma_c@>>> \Gamma@>>>1\\
  @.@V \Id VV@V \rho VV@V \tilde\rho VV\\
  1@>>> Z' @>>>\Aut_{\Gamma_c,\tau}(E)@>>>\Aut_{\Gamma,\tau}(E')@>>>1.
  \end{CD}
  \end{displaymath}

This completes the proof.
\end{proof}


Recall that a $G$-bundle $E$ is said to be {\bf simple} if $\Aut(E)\cong Z$.
We have the following.

\begin{proposition}\label{fix-simple}
Let $E$ be a simple $G$-bundle over $X$ such that 
\begin{equation}\label{condition-on-E}
E\cong\begin{cases} 
\gamma^*E&\mbox{for every}\;\; \gamma\in \Gamma^+\\
\gamma^*\tau(E)&\mbox{for every}\;\; \gamma\in \Gamma^-.
\end{cases}
\end{equation}
Then $E$ admits a  $(\Gamma,\tau,c)$-equivariant structure with 
$c\in Z^2_\tau(\Gamma,Z)$.
\end{proposition}
\begin{proof}
Condition (\ref{condition-on-E}) implies the existence of
an exact sequence 
$$
1\to \Aut(E)\to \Aut_{\Gamma,\tau}(E) \to \Gamma\to 1. 
$$
Now, since $E$ is simple $\Aut(E)\cong Z$ and hence we have an extension
$$
1\to Z\to \Aut_{\Gamma,\tau}(E) \to \Gamma\to 1. 
$$
This extension is determined by a cocycle $c\in Z^2_\tau(\Gamma,Z)$, which is 
precisely the obstruction to having a  $(\Gamma,\tau)$-equivariant structure
on $E$, i.e. a homomorphism $\Gamma\to \Aut_{\Gamma,\tau}(E)$ splitting the 
exact sequence. However we have a twisted homomorphism  of $\Gamma$
in $\Aut_{\Gamma,\tau}(E)$ with cocycle $c$, that is, a 
$(\Gamma,\tau,c)$-equivariant structure.
\end{proof}

\subsection{Isotropy subgroups associated to a 
$(\Gamma,\tau,c)$-equivariant structure}
We will assume that $\Gamma^+\neq \{1\}$.
The case $\Gamma^+=\{1\}$ has been 
extensively studied in \cite{biswas-garcia-prada,biswas-garcia-prada-hurtubise,
biswas-garcia-prada-hurtubise2} and corresponds to the study of twisted real 
structures on $E$.

Let $x\in X$, and 
$$
\Gamma_x:=\{\gamma\in \Gamma^+\;\;:\;\;\gamma(x)=x\}
$$
be the corresponding {\bf isotropy subgroup}.
Let $\PPP=\{x\in X \, : \,  \Gamma_x\neq \{1\}\}$.

The following is well-known (see \cite{nasatyr-steer} for example).

\begin{proposition}
(1) $\PPP$ consists of a finite number of points 
$\{ x_1, \ldots, x_r \} \subset X$.

(2) For  each $x_i\in \PPP$, 
$\Gamma_{x_i}$ is cyclic.  
\end{proposition} 

Let $c_x\in Z^2(\Gamma_x,Z')$ be the restriction of $c$ to $\Gamma_x$
(note that the action of $\Gamma_x$ on $Z'$ is trivial since 
$\Gamma_x\subset\Gamma^+$). Define the  $c_x$-{\bf twisted character variety} of  $\Gamma_x$ in $G$
as the set
$$ 
R_{c_x}(\Gamma_x,G):= \Hom_{c_x}(\Gamma_x,G)/G,
$$
where
$$
\Hom_{c_x}(\Gamma_x,G):= \{\sigma:\Gamma_x\to G\;\;|\;\;
\sigma(\gamma\gamma')=c_x(\gamma,\gamma')\sigma(\gamma)\sigma(\gamma')\},
$$
and two elements $\sigma,\sigma'\in \Hom_{c_x}(\Gamma_x,G)$ 
are equivalent under the action of
$G$ if  
$$
\sigma'(\gamma)=g^{-1}\sigma(\gamma)g\;\;\mbox{for some}\;\; g\in G.
$$

\begin{proposition}\label{pseudorep}
A $(\Gamma,\tau,c)$-equivariant structure 
on a $G$-bundle $\pi: E\to X$ defines for every $x\in \PPP$ an element
$\sigma_x\in  R_{c_x}(\Gamma_x,G)$.
\end{proposition}

\begin{proof}

For each $x\in \PPP $ and $e\in E$ such that $\pi(e)=x$, a straightforward computation shows that the map
$\sigma_e: \Gamma_x \to G$ given by 

\begin{equation}\label{eqn:pseudo-hom-def}
\widetilde{\gamma}(e)=e\sigma_e(\gamma)
\end{equation}
defines an element in $\Hom_{c_x}(\Gamma_x,G)$. Moreover,  if
$e'\in \pi^{-1}(x)$, with $e'=eg$ for $g\in G$, then
$\sigma_{e'}(\gamma)=g^{-1}\sigma_e(\gamma)g$, proving the assertion.
\end{proof}


\begin{remark}
The composition of $\sigma_e$ with 
the projection $G\to G/Z'$, defines a homomorphism $\rho_e:\Gamma_x\to G/Z'$. 
Of course, $c$ restricted to $\Gamma^+$ is trivial,
i.e., if the restriction of the action of $\Gamma$ to $\Gamma^+$ defines 
a genuine $\Gamma^+$-equivariant structure on $E$, then 
$\sigma_e$ itself is a homomorphism, and $\sigma_x$ is an element of the 
character variety $R(\Gamma_x,G):=\Hom(\Gamma_x,G)/G$. 
\end{remark}

The  following is clear.
\begin{proposition}
Let $c$ and $c'$ be $2$-cocycles in $Z^2_\tau(\Gamma,Z')$. Let 
$\sigma_x \in R_{c_x}(\Gamma_{x},G)$ and $\sigma'_x \in 
R_{c'_x}(\Gamma_{x},G)$  be 
corresponding classes. Then the projections of $\sigma_x$ and
$\sigma'_x$ in $R(\Gamma_x,G/Z')$ coincide.
\end{proposition}  

The next result shows that the $\Gamma$ action defines a bijection between spaces 
of twisted representations of isotropy groups over points in $X$ related by the action of $\Gamma$.

\begin{proposition}
(1)The action of $\Gamma$ on $X$ induces an action of $\Gamma$ (and in particular of
$\Gamma^+$) on $\PPP$. 

(2) Let 
$\QQQ=\PPP/\Gamma^+$. If $x$ and $x'$ are in the same class in $\QQQ$ there is an isomorphism
$R_{c_x}(\Gamma_{x},G)\cong R_{c_{x'}}(\Gamma_{x'},G)$ (as pointed sets)
under which $\sigma_x$ and $\sigma_{x'}$ are in correspondence. This isomorphism 
induces a canonical isomorphism $R(\Gamma_x, G/ Z')\cong R(\Gamma_{x'}, G/Z')$. 

(3) If two points $y,y'\in \QQQ$ are in correspondence 
under the residual action of $\Z/2\cong \Gamma/\Gamma^+$ on $\QQQ$, then 
$R_{c_x}(\Gamma_{x},G)$ and $R_{c_{x'}}(\Gamma_{x'},G)$ are in a bijective correspondence given by
$\sigma_x\mapsto \tau\sigma_{x'}$ for any representatives 
$x,x'\in \PPP$ of  $y,y'\in \QQQ$ respectively.
\end{proposition}

\begin{proof}
Statement (1) follows from the fact that two points on $X$ connected by the action of $\Gamma$ must have conjugate isotropy subgroups.
To prove (2), if two points $x,x'\in \PPP$ are in the same class in $\QQQ$, then there exists $\gamma_0 \in \Gamma^+$ such that 
$x' = \gamma_0 \cdot x$ and so  $\Gamma_{x'}=\gamma_0 \Gamma_x\gamma_0^{-1}$. Let $\tilde{\gamma}_0$ denote the lift of $\gamma_0$ to 
$\Aut_{\Gamma^+}(E)$, where $\Aut_{\Gamma^+}(E)$ is the preimage of $\Gamma^+$ in the exact sequence (\ref{exact-aut}).
Given any $e_x$ in the fibre $E_x$, let $e_{x'} := \tilde{\gamma}_0 (e_x)$. For any $\gamma \in \Gamma_{x}$, let $\gamma' = \gamma_0 \gamma \gamma_0^{-1}$ be the corresponding element of $\Gamma_{x'}$ and let $\tilde{\gamma}$ and $\tilde{\gamma}' = \tilde{\gamma}_0 \tilde{\gamma} \tilde{\gamma}_0^{-1}$ denote the respective lifts to $\Aut_\Gamma(E)$. Using \eqref{eqn:pseudo-hom-def} we have
\begin{equation*}
\tilde{\gamma}(e_x) = e_x \sigma_{e_x}(\gamma) \quad \text{and} \quad \tilde{\gamma}'(e_{x'}) = e_{x'} \sigma_{e_{x'}}(\gamma') .
\end{equation*}
Therefore
\begin{equation*}
e_{x'} \sigma_{e_{x'}}(\gamma') = \tilde{\gamma}' (e_{x'}) = \tilde{\gamma}_0 \tilde{\gamma} \tilde{\gamma}_0^{-1} (e_{x'}) = \tilde{\gamma_0} \tilde{\gamma}(e_x) = \tilde{\gamma}_0 e_x \sigma_{e_x}(\gamma) = e_{x'} \sigma_{e_x}(\gamma)
\end{equation*}
and so $\sigma_{e_{x'}}(\gamma') = \sigma_{e_{x'}} (\gamma_0 \gamma \gamma_0^{-1}) = \sigma_{e_x}(\gamma)$. Therefore we see that $\sigma_{e_x}$ determines $\sigma_{e_{x'}}$ and vice versa, and so the same is true for $\sigma_{x}$ and $\sigma_{x'}$. 

Therefore, a choice of $\gamma_0$ such that $x' = \gamma_0 \cdot x$ determines a bijection $R_{c_x}(\Gamma_x, G) \rightarrow 
R_{c_{x'}}(\Gamma_{x'}, G)$  sending  
$\sigma \mapsto \sigma'$ with $\sigma'(\gamma') := \sigma(\gamma_0^{-1} \gamma' \gamma_0)$, and this bijection maps $\sigma_x$ to $\sigma_{x'}$.

An element  $\sigma\in \Hom_{c_x}(\Gamma_x, G)$ descends to a homomorphism $\bar{\sigma} : \Gamma_x \rightarrow G / Z'$.  
The bijection $\sigma \mapsto \sigma'$ defined above induces a map $\bar{\sigma} \mapsto \bar{\sigma}'$ defined by
\begin{equation*}
\bar{\sigma}' (\gamma') := \bar{\sigma}(\gamma_0^{-1} \gamma' \gamma_0) .
\end{equation*}
Given any other choice $\gamma_1$ such that $\Gamma_{x'} = \gamma_1 \Gamma_x \gamma_1^{-1}$, we have $\gamma_1 \gamma_0^{-1} \in \Gamma_{x'}$ and so (since $\bar{\sigma}$ is a homomorphism) for any $\gamma' \in \Gamma_{x'}$ we have
\begin{equation*}
\bar{\sigma}(\gamma_1^{-1} \gamma' \gamma_1) = \bar{\sigma}(\gamma_1^{-1} \gamma_0) \bar{\sigma}(\gamma_0^{-1} \gamma' \gamma_0) \bar{\sigma}(\gamma_1^{-1} \gamma_0)^{-1} .
\end{equation*}
Therefore the conjugacy class of $\bar{\sigma}'$ in $R(\Gamma_{x'}, G / Z')$ is well-defined and independent of the choice of $\gamma_0$ such that $\Gamma_{x'} = \gamma_0 \Gamma_x \gamma_0^{-1}$.

(3) follows from a straightforward computation.

\end{proof}

\begin{remark}
Note that if in (3) $y\in \QQQ$ is a fixed point under the residual action of $\Z/2$ then the twisted representation $\sigma_x$ must lie 
in the real group $G^\tau$.
\end{remark}

\subsection{Twisted $\Gamma$-equivariant structures on
associated vector bundles}\label{vector-bundles}
Let now $V$ be a rank $n$  holomorphic complex vector bundle over $X$. 
Let $\tau_{\VV}$ be a conjugation on the fibre $\VV$ of $V$. Consider the 
action of $\Gamma$ on $\C^*$ given by
\begin{equation}\label{action-on-C}
 z^\gamma= \begin{cases}
z&\mbox{if}\;\;  \gamma\in \Gamma^+\\
\overline{z}&\mbox{if}\;\; \gamma\in \Gamma^-.
\end{cases}
\end{equation}
Given a cocycle  $c\in Z^2(\Gamma,\C^*)$ for this action, similarly as for $G$-bundles 
one can define  a $(\Gamma,\tau_\VV,c)$-equivariant structure on $V$ as a  
$c$-twisted representation of $\Gamma$ in $\Aut_{\Gamma,\tau_\VV}(V)$, where  $\Aut_{\Gamma,\tau_\VV}(V)$
is defined in a similar fashion to the $G$-bundle case.

Now, let $E$ be a principal $G$-bundle and  $\rho:G\to GL(\VV)$ a representation of
$G$ in a complex vector space $\VV$. Consider the associated vector bundle
$V:=E(\VV)$. Let $\tau$ and $\tau_\VV$  be conjugations of $G$ and $\VV$, respectively
Let $c\in Z^2_\tau(\Gamma,Z')$ and $c_\rho\in Z^2(\Gamma,\C^*)$ be the cocycle 
induced by $\rho|_{Z'}:Z'\to \C^*\cong Z(\GL(\VV))$. If $\rho$ is compatible with the conjugations
 $\tau$ and $\tau_\VV$, then there is a homomorphism
$\Aut_{\Gamma,\tau}(E)\to \Aut_{\Gamma,\tau_\VV}(V)$, and it is clear that
a $(\Gamma,\tau,c)$-equivariant structure on $E$ defines a 
$(\Gamma,\tau_\VV,c_\rho)$-equivariant structure on $V$. In particular if 
$Z'\subset \ker \rho$, then  $c_\rho$ is trivial and hence  we obtain
$(\Gamma,\tau_\VV)$-equivariant structure on $V$. If moreover $\Gamma=\Gamma^+$,
this is a genuine $\Gamma$-equivariant structure on $V$.

\section{ Twisted equivariant structures on Higgs bundles}

In this section $X$ is a compact Riemann surface of genus bigger than one,
$\Gamma$ is a subgroup of $\Aut(X)$, the group of holomorphic or 
antiholomorphic automorphisms of $X$, and $G$ is 
a  connected {\bf real} reductive Lie group. As in Section \ref{higgs-bundles},
we fix a maximal compact 
subgroup $H$ of $G$.  The Lie algebra $\lieg$ of $G$ is equipped with 
an involution $\theta$ 
that gives the Cartan decomposition  $\lieg=\lieh \oplus \liem$,
where $\lieh$ is the Lie algebra of $H$.
We choose a complex conjugation $\tau$ of $H^\C$, and a conjugation 
$\tau_{\liem^\C}$ of $\liem^\C$, such that the isotropy representation 
$\iota: H^\C\to \Aut(\liem^\C)$ is compatible with $\tau$ and $\tau_{\liem^\C}$.
This is the case if for example $G$ is a real form of a complex reductive
group $G^\C$ and we choose a complex conjugation $\widetilde{\tau}$ of $G^\C$ 
commuting with the Cartan involution of $G$ extended to $G^\C$. The conjugation
$\tau$ and  $\tau_{\liem^\C}$ induced by  $\widetilde{\tau}$ satisfy the 
compatibility condition with $\iota$. As proved by Cartan, we can always choose
a compact conjugation $\widetilde{\tau}$ commuting with the Cartan involution.
This is the choice which is relevant in connection to the study of $\Gamma$ on 
the moduli space of representations $\calR(S,G)$. 
 
\subsection{Twisted $\Gamma$-equivariant structures on $G$-Higgs bundles}

Let $(E,\varphi)$ be a $G$-Higgs bundle over $X$. We will define now 
twisted $\Gamma$-equivariant structures on $(E,\varphi)$. To do this,
let $Z=Z(H^\C)$, and let $Z'\subset Z$ be a subgroup invariant under the 
action of $\tau$.  Choose a $2$-cocycle $c\in Z^2_\tau(\Gamma,Z')$. 
Recall from Section \ref{vector-bundles}, that this defines a $2$-cocycle 
$c_\iota\in Z^2(\Gamma,\C^*)$,
via the isotropy representation $\iota: H^\C\to \GL(\liem^\C)$. 

If the  $H^\C$-bundle $E$ is  equipped with a  
$(\Gamma,\tau,c)$-equivariant structure, from Section 
\ref{vector-bundles}, the vector bundle $E(\liemc)$ inherits a 
$(\Gamma,\tau_{\liemc},c_\iota)$-equivariant structure. On the other hand, 
the canonical bundle $K$ over $X$ has a natural 
$(\Gamma,\tau_\C)$-equivariant structure induced by the action of $\Gamma$ on $X$.  We conclude then that the bundle $E(\liemc)\otimes K$ has a 
$(\Gamma,\tau_{\liemc},c_\iota)$-equivariant structure (where we are omitting 
$\tau_\C$ in the notation). In fact, we will abuse notation, and use $\tau$ to 
refer to both $\tau$ and $\tau_{\liemc}$ in the sequel.

A {\bf $(\Gamma,\tau,c)$-equivariant structure on $(E,\varphi)$}
is a $(\Gamma,\tau,c)$-equivariant structure on $E$, such that 
for every $\gamma\in \Gamma$ the following diagram commutes: 
$$
\begin{matrix}
E(\liemc)\otimes K & \stackrel{\widetilde{\gamma}}{\longrightarrow} & 
E(\liemc)\otimes K\\
~\Big\uparrow\varphi && ~\,\text{  }~\,\text{ }\Big\uparrow \varphi\\
X & \stackrel{\gamma}{\longrightarrow} & X
\end{matrix},
$$
where $\widetilde{\gamma}$ is the collection of maps defining the 
$(\Gamma,\tau, c_\iota)$-equivariant structure on $E(\liem^\C)\otimes K$ 
defined above.  


The notion  of stability for $G$-Higgs bundles given in Section 
\ref{higgs-bundles}
(see \cite{garcia-prada-gothen-mundet}) can be 
extended in a natural way to  $G$-Higgs bundles equipped with  
twisted equivariant structures. This is done in a similar way to that in the study of pseudoreal Higgs bundles 
\cite{biswas-garcia-prada-hurtubise,biswas-calvo-garcia-prada}. To explain this, 
we consider the adjoint bundle of groups associated to the $H^\C$-bundle $E$.
This is defined as $\Ad(E)=E\times_{H^\C} H^\C$, where $H^\C$ acts on 
$E\times H^\C$ by
$$
(e,g)\cdot h=(eh, h^{-1}gh)\;\;\mbox{for}\;\; h,g\in H^\C\;\;\mbox{and}\;\; e\in E.
$$
We define now for $\gamma\in \Gamma$ a map 
$\widetilde{\gamma}^{\Ad}:E\times H^\C \to E\times H^\C$ given by 
\begin{equation}\label{maps-adjoint}
\widetilde{\gamma}^{\Ad}(e,g)=(\widetilde{\gamma}(e), 
\gamma(g))\;\;\mbox{for}\;\;
\gamma\in \Gamma, e\in E\;\;\mbox{and}\;\; g\in H^\C,
\end{equation}
where, recall
$$
\gamma(g)= \begin{cases}
g&\mbox{if}\;\;  \gamma\in \Gamma^+\\
\tau(g)&\mbox{if}\;\; \gamma\in \Gamma^-.
\end{cases}
$$
\begin{proposition}
The maps $\{\widetilde{\gamma}^{\Ad}\}_{\gamma\in \Gamma}$
 define a $\Gamma$-equivariant structure on $\Ad(E)$.
\end{proposition}
\begin{proof}
First we have to check that the maps (\ref{maps-adjoint}) descend to $\Ad(E)$.
Let $e\in E$ and $g,h\in H^\C$. We have
$$
\widetilde{\gamma}^{\Ad}(eh,h^{-1}gh)=(\widetilde{\gamma}(eh),\gamma(h^{-1}gh))=
(\widetilde{\gamma}(e)\gamma(h),\gamma(h)^{-1}\gamma(g)\gamma(h)),
$$ 
but  $(\widetilde{\gamma}(e)\gamma(h),\gamma(h)^{-1}\gamma(g)\gamma(h))$ is equivalent
to $(\widetilde{\gamma}(e),\gamma(g))$ via the action of $\gamma(h)$, and we thus  have
well-defined maps $\widetilde{\gamma}^{\Ad}: \Ad(E)\to\Ad(E)$.

Since the action of the centre of $H^\C$ by inner automorphisms on $H^\C$ is trivial, the $2$-cocycle $c$ has no effect,  and one checks that 
$$
\widetilde{\gamma\gamma'}^{\Ad}=\widetilde{\gamma}^{\Ad}\widetilde{\gamma'}^{\Ad}.
$$
  
\end{proof}

As studied in 
\cite{biswas-garcia-prada-hurtubise,biswas-calvo-garcia-prada}, 
the main change in the
definition of stability for a  $G$-Higgs bundle equipped with a 
twisted equivariant structure, in relation to the usual stability condition 
for the underlying $G$-Higgs bundle  given in Section 
\ref{higgs-bundles}, is that we must consider only holomorphic
reductions $E_{P_s} \,\subset\, E$ of $E$ from $H^\C$ to $P_s$ such that 
\begin{equation}\label{eq:stabilitycondition}
\widetilde{\gamma}^{\Ad}(\Ad(E_{P_{s}}))\,=\,\Ad(E_{P_{s}})\;\;\mbox{for every}\;\;
\gamma\in \Gamma.
\end{equation}

To define the moduli space of twisted $\Gamma$-equivariant $G$-Higgs bundles,
 we fix the cocycle
$c\in Z^2_\tau(\Gamma,Z')$ and the elements 
$\sigma_i\in R_{c_{x_i}}(\Gamma_{x_i},H^\C)$ for every point $x_i\in\PPP$ defined by 
Proposition \ref{pseudorep}. We will need at some point the projection of 
$\sigma_i$ in $R(\Gamma_{x_i},H^\C/Z')$ that we will denote by $[\sigma_i]$. 
Let $\sigma=(\sigma_1,\cdots, \sigma_r)$.
We  define $\cM(X,G,\Gamma,\tau,c)$ to be the  
{\bf moduli space  of polystable 
$(\Gamma,\tau,c)$-equivariant $G$-Higgs bundles}. 
An analytic construction of these spaces can be given using slices.
The subvariety of $\cM(X,G,\Gamma,\tau,c)$  with fixed 
classes $\sigma$ will be denoted by $\cM(X,G,\Gamma,\tau,c,\sigma)$. 

We will assume now that the conjugation $\tau$ of $H^\C$ commutes with the
compact conjugation of $H^\C$ defining a maximal compact subgroup 
$H\subset H^\C$, in other words, that $H$ is invariant under $\tau$, This is indeed a condition satisfied in connection to 
our application to the study of fixed points in the moduli space of
representations of the fundamental group of the surface in $G$.
Under this assumption we have the following.

\begin{proposition}\label{action}
Let $E$ be a $H^\C$-bundle over $X$ equipped with a 
$(\Gamma,\tau,c)$-equivariant structure, where $c\in Z^2_\tau(\Gamma,Z')$, 
and let $H\subset H^\C$ be a maximal compact subgroup preserved 
by $\tau$, so that $Z'\subset H$. Then the twisted equivariant structure on $E$ 
induces a group action of $\Gamma$  on the space of reductions of structure
group of $E$ to $H$. 
\end{proposition}
\begin{proof}
Note that since $\tau$ preserves $H$, the action of $\Gamma$ on $H^\C$ induces 
an action of $\Gamma$ on $M:=H^\C/H$. So we have actions of $H^\C$ and $\Gamma$ 
on $M$ satisfying that
\begin{equation}\label{action-compatible}
\gamma\cdot(gm)=\gamma(g)(\gamma\cdot m)\;\;\mbox{for}\;\;\gamma\in\Gamma, g\in H^\C\;\;
\mbox{and}\;\; m\in M:=H^\C/H.
\end{equation} 

Recall that a reduction of structure group of $E$ to $H$ is a section 
of $E(M)$, the 
$M$-bundle associated to $E$ via the natural action of $H^\C$ on 
$M:= H^\C/H$ on the  left. Such a section is equivalent to a 
$H^\C$-antiequivariant map  $\psi:E\to M$, i.e.,
$\psi(eg)=g^{-1}\psi(e)$, for $e\in E$ and $g\in H^\C$. 
For such a map $\psi$, and $\gamma\in \Gamma$, define a map
$\gamma\cdot\psi: E\to M$ given by
$$
(\gamma\cdot \psi)(e):=\gamma^{-1}\cdot \psi(\widetilde{\gamma}(e)),
$$
where $\widetilde{\gamma}$ is given by the $(\Gamma,\tau,c)$-equivariant
structure on $E$. We need to check first that $\gamma\cdot \psi$ is 
$H^\C$-antiequivariant.
$$
(\gamma\cdot \psi)(eg)=\gamma^{-1}\cdot\psi(\widetilde{\gamma}(eg))=
\gamma^{-1}\cdot \psi(\widetilde{\gamma}(e)\gamma(g))=
\gamma^{-1}\cdot(\gamma(g)^{-1}\psi(\widetilde{\gamma}(e))).
$$
But from (\ref{action-compatible}), we deduce that
$$
\gamma^{-1}\cdot(\gamma(g^{-1})\psi(\widetilde{\gamma}(e)))=
g^{-1}(\gamma^{-1}\cdot \psi(\widetilde{\gamma}(e)))=
g^{-1}((\gamma\cdot \psi)(e)),
$$
proving the antiequivariance of $\gamma\cdot\psi$. To check that this defines 
an action of $\Gamma$ on the space of sections of $E(M)$, we consider for 
$\gamma,\gamma'\in \Gamma$
$$
\begin{array}{lll}
((\gamma\gamma')\cdot \psi)(e)&=&
(\gamma\gamma')^{-1}\cdot\psi(\widetilde{\gamma\gamma'}(e)) \\
& =&
(\gamma')^{-1}\gamma^{-1}\cdot(\psi(\widetilde{\gamma}(\widetilde{\gamma'}(e))c(\gamma,\gamma'))  \\ 
&   = &(\gamma')^{-1}\gamma^{-1}\cdot (c(\gamma,\gamma')^{-1}
\psi(\widetilde{\gamma}(\widetilde{\gamma'}(e))).
\end{array}
$$
But, since $Z'\subset H$, the action of  $c(\gamma,\gamma')^{-1}$ is trivial,
and we see that
$$
(\gamma\gamma')\cdot \psi=\gamma'\cdot(\gamma\cdot \psi),
$$
completing the proof.
\end{proof}

Given a $(\Gamma,\tau,c)$-equivariant $G$-Higgs bundle $(E,\varphi)$ such that
$Z'\subset H$ and $H$ is invariant by $\tau$, by Proposition \ref{action} we can consider the action of 
$\Gamma$ on 
the space of metrics on $E$, that is on the space of sections of $E(H^\C/H)$.
The analysis done for the Hitchin--Kobayashi correspondence given in Section
\ref{higgs-bundles} can be extended to this equivariant situation  
to prove the following (see \cite{biswas-garcia-prada-hurtubise,biswas-calvo-garcia-prada}).

\begin{theorem}\label{theo:hk-twisted-pairs-equivariant}
Let $(E,\varphi)$ be a $G$-Higgs bundle over a Riemann surface $X$
equipped with a $(\Gamma,\tau,c)$-equivariant structure, with cocycle 
$c\in Z^2(\Gamma,Z')$, where  $Z'\subset H$, and $H$ is invariant by $\tau$.
 Then $(E,\varphi)$ is polystable as a $(\Gamma,\tau,c)$-equivariant
Higgs bundle if and only if 
there exists a $\Gamma$-invariant reduction $h$ of the structure group of $E$ from $H^\C$ to
$H$,   such that
\begin{equation}\label{eq:Hitchin-Kobayashi-h}
F_h - [\varphi,\tau_h(\varphi)]=0 . \\
\end{equation}
\end{theorem}

From Theorems \ref{theo:hk-twisted-pairs-equivariant}
and \ref{theo:hk-twisted-pairs} we conclude the following.

\begin{proposition}\label{forgetful}
Let $Z'\subset Z\cap H$ and $c\in Z^2(\Gamma,Z')$, and assume that
$H$ is invariant by $\tau$. Then 
the forgetful map defines a morphism
$\cM(X,G,\Gamma,\tau,c,\sigma)\to \cM(X,G)$.
\end{proposition}

\subsection{$\Gamma$-action on the moduli space of $G$-Higgs bundles}

Consider the action of $\Gamma$ 
on the  moduli space of $G$-Higgs bundles $\cM(X,G)$ given by the rule:  
$$
\gamma\cdot(E,\varphi)= \begin{cases}
(\gamma^*E,\gamma^*\varphi) & \mbox{if} \;\; \gamma\in \Gamma^+\\
(\gamma^*\tau(E),\gamma^*\tau(\varphi)) & \mbox{if} \;\; \gamma\notin \Gamma^+.
\end{cases}
$$
We have the following.

\begin{theorem}\label{fixed-points-theorem}
 Let $Z'\subset Z\cap H$ and  assume that $H$ is invariant by $\tau$. Let
$\widetilde{\cM}(X,G,\Gamma,\tau,c,\sigma)$
be the image of
the  morphism in Proposition \ref{forgetful}. Then

(1) If $c$ and $c'$ are
cohomologous cocycles in $Z^2_\tau(\Gamma,Z')$ 
$$
\widetilde{\cM}(X,G,\Gamma, \tau, c,\sigma)= \widetilde{\cM}(X,G,\Gamma,\tau, c',\sigma'). 
$$

(2) For any $Z'\subset Z$, any $\sigma$,  and any cocycle $c\in Z^2_\tau(\Gamma,Z')$
$$
\widetilde{\cM}(X,G,\Gamma, \tau,c,\sigma)\subset \cM(X,G)^\Gamma.
$$

(3) Let $\cM_*(X,G)\subset \cM(X,G)$ be  the subvariety of $G$-Higgs
bundles which are stable and simple and let $Z'=Z\cap\ker \iota$, then 
$$
\cM(X,G)_*^\Gamma\subset \bigcup_{[c]\in H^2_\tau(\Gamma,Z'), 
\sigma: [\sigma_i]\in R(\Gamma_{x_i},H^\C/Z')} 
\widetilde{\cM}(X,G,\Gamma,\tau,c,\sigma).
$$
\end{theorem}

\begin{proof}
To prove (1), we consider the function $f:G\to Z'$ such that $c$ and $c'$
are related by (\ref{coboundary}). This function defines an automorphism
of a $G$-Higgs bundle $(E,\varphi)$ which sends the twisted equivariant 
structure with cocycle $c$ and isotropy $\sigma$ to a twisted equivariant 
structure with cocycle $c'$ and isotropy $\sigma'$. The proof of (2) follows 
immediately 
from the definition of twisted equivariant structure. The proof of (3) follows
a similar argument to that of Proposition \ref{fix-simple}:
The condition $(E,\varphi)\cong (\gamma^*E,\gamma^*\varphi)$  
if $\gamma\in \Gamma^+$ or  $(E,\varphi)\cong (\gamma^*\tau(E),\gamma^*
\tau(\varphi)$   if $\gamma\in \Gamma^-$  
implies the existence of 
an exact sequence 
$$
1\to \Aut(E,\varphi)\to \Aut_{\Gamma,\tau}(E,\varphi) \to \Gamma\to 1, 
$$
where  $\Aut(E,\varphi)$ is the group of automorphisms of $(\Aut(E,\varphi)$
covering the identity and and  $\Aut_{\Gamma,\tau}(E,\varphi)$ 
is the subgroup of $\Aut_{\Gamma,\tau}(E)$ defined by \ref{twisted-autos}
defined by elements which send $\varphi$ to $\varphi$ if $\gamma\in \Gamma^+$ and $\varphi$ to $\tau(\varphi)$ if $\gamma\in \Gamma^-$.

Since we are assuming that $(E,\varphi)$ is simple 
$\Aut(E,\varphi)\cong Z'=Z\cap \ker\iota$ and hence we have an extension
$$
1\to Z'\to \Aut_{\Gamma,\tau}(E,\varphi) \to \Gamma\to 1. 
$$
This extension defines a cocycle $c\in Z^2_\tau(\Gamma,Z')$,  and a 
$c$-twisted homomorphism $\Gamma\to \Aut_\Gamma(E,\varphi)$ 
with cocycle $c$, i.e., a $(\Gamma,\tau,c)$-equivariant
structure on $(E,\varphi)$. It follows from (1)  that the 
union should run over  $[c]\in H^2_\tau(\Gamma,Z')$ and  
$[\sigma_i]\in R(\Gamma_{x_i},H^\C/Z')$, where, recall that  $[\sigma_i]$ is the 
projection of $\sigma_i$ in $R(\Gamma_{x_i},H^\C/Z')$.
\end{proof}

\section{Equivariant structures and parabolic Higgs bundles}

As in the previous section, let $X$ be a compact Riemann surface, let 
$\Gamma \subset \Aut(X)$ be a finite subgroup.
We will assume here  that $\Gamma=\Gamma^+$.
let $Y:=X/\Gamma$ and $\pi_Y:X\to Y$ be the associated ramified covering map.  The set of points 
$\PPP\subset X$ maps by $\pi_Y$ to a set $\SSS \subset Y$. 
In this section we establish a correspondence between $\Gamma$-equivariant
$G$-Higgs bundles over $X$ and parabolic $G$-Higgs bundles over $Y$ with parabolic
points $\SSS$. This extends the well-known correspondences for vector bundles
\cite{mehta-seshadri,furuta-steer,nasatyr-steer,biswas,andersen-masbaum,
andersen-grove}, and principal bundles \cite{teleman-woodward,balaji-seshadri}. 
In particular this implies that if $Z'=Z\cap\ker\iota$ and a $G$-Higgs bundle
$(E,\varphi)$ is equipped with a $(\Gamma,c)$-equivariant structure with 
$c\in Z^2(\Gamma,Z')$, then  $(E',\varphi)$  with $E':=E/Z'$ is a $G'=G/Z'$-Higgs
bundle with a $\Gamma$-equivariant structure and hence is in correspondence
with a parabolic $G'$-Higgs bundle over $Y$.   
It would be very interesting to give a parabolic description of the 
twisted equivariant structure on $(E,\varphi)$.

\subsection{Parabolic $G$-Higgs bundles}

In this section $Y$ is a compact Riemann surface,
and $G$  is a connected real  reductive Lie group.
We keep the same notation as in the previous sections
for a maximal compact subgroup, isotropy representation, etc.

Let $T\subset H$ be a Cartan subgroup, and $\liet$ be its Lie
algebra. We consider a Weyl alcove $\AAA\subset \liet$
(see \cite{biquard-garcia-prada-mundet}). Recall that if $W$ is the Weyl group we have
$$
\AAA \cong T/W\cong \Conj(H),
$$
where $\Conj(H)$ is the set of conjugacy classes of $H$.
Note that in contrast to the definition of alcove in \cite{biquard-garcia-prada-mundet},
here $\AAA$ may contain some walls so that it is a fundamental
domain for the action of the affine Weyl group.

Let $\SSS=\{y_1, \ldots,  y_s\}$ be a finite set of distinct points
of $Y$ and $D=y_1+\cdots + y_s$ be the corresponding effective
divisor.

An element $\alpha\in \sqrt{-1}\AAA$ defines a parabolic subgroup
of $P_\alpha\subset H^\C$ given by (\ref{parabolic}). Fix for every point
$y_i\in \SSS$ an element $\alpha_i\in  \sqrt{-1}\AAA$, and denote
$\alpha=(\alpha_1,\cdots,\alpha_s)$.

A {\bf parabolic $G$-Higgs bundle over $(Y,\SSS)$ with weights $\alpha$} 
is a pair $(E,\varphi)$ consisting of a holomorphic $H^\C$-bundle $E$ over $Y$
equipped with a reduction of $E_{y_i}$ to $P_{\alpha_i}$ and $\varphi$ is a 
holomorphic section of $PE(\liem^\C)\otimes K(D))$, where $PE(\liem^\C)$ is 
the sheaf of parabolic sections of $E(\liem^\C)$  (see \cite{biquard-garcia-prada-mundet} for details). There are notions of 
stability, semistability and polystability similar to the ones 
we have already seen in previous sections (\cite{biquard-garcia-prada-mundet}).

To define a moduli space one has to fix for every point
 $y_i\in \SSS$ the projection $\LLL_i$  of the residue of
$\varphi$ in $\liem_{\alpha_i}^0/L_{\alpha_i}$, where $\liem_{\alpha_i}^0$ and 
$L_{\alpha_i}$ are defined as in Section \ref{higgs-bundles}.
Denote $\LLL=(\LLL_1,\cdots,\LLL_s)$. 
We define $\cM(Y,\SSS,G,\alpha,\LLL)$ to be the 
{\bf moduli space of parabolic $G$-Higgs bundles on $(Y, \SSS)$ with weights 
$\alpha=(\alpha_1,\cdots,\alpha_s)$ and
residues  $\LLL=(\LLL_1,\cdots,\LLL_s)$}.
 
\subsection{$\Gamma$-equivariant Higgs bundles and parabolic Higgs  bundles}
\label{sec:gamma-hol}






In this section we describe the correspondence between parabolic $G$-Higgs 
bundles on $Y$ and $\Gamma$-equivariant $G$-Higgs bundles on $X$. For 
holomorphic vector bundles over a compact Riemann surface, this correspondence originated in \cite{furuta-steer} and was generalised to higher dimensions in \cite{biswas}. The extension to Higgs vector bundles was carried out in \cite{nasatyr-steer}, and for holomorphic principal bundles this correspondence is contained in \cite{teleman-woodward} and \cite{balaji-seshadri}. 

First we begin with the data of a compact Riemann surface $X$ and a finite subgroup $\Gamma \subset \Aut(X)$ consisting entirely of holomorphic automorphisms. Applying the smoothing process of \cite[Sec. 2]{boden} to the orbifold $X / \Gamma$ determines a compact Riemann surface $Y$ and a holomorphic map $\pi_Y : X \rightarrow Y$ such that $\Gamma$ is the group of deck transformations of the ramified cover $\pi$. Let $\{ x_1, \ldots, x_r \}$ denote the ramification points of $\pi$ and let $D = y_1 + \cdots + y_s$ denote the branch divisor. Each ramification point $x_j$ has a non-trivial isotropy group denoted $\Gamma_{x_j} \subset \Gamma$ which is cyclic of order $N_j$. Let $N = | \Gamma |$ denote the order of the ramified cover $\pi_Y : X \rightarrow Y$. 

Let $E \rightarrow X$ be a principal $H^\C$ bundle, and choose a lift of $\Gamma$ to the group of $C^\infty$ automorphisms of $E$. Via this lift, each isotropy group $\Gamma_{x_j} \cong \mathbb{Z}_{N_j}$ acts on the fibre $E_{x_j}$ which determines a representation $\sigma_j \in R(\Gamma_{x_j}, H^\C)$ (note that since we are considering equivariant rather than twisted equivariant bundles then the cocycle $c \in Z^2(\Gamma, Z')$ is trivial). 

Let $\CCC_{x_j} \in \Conj(H)$ denote the conjugacy class of the generator $\gamma_{x_j}$ of $\Gamma_{x_j}$, which is determined by the representation $\sigma_j$. Under the bijection between $\Conj(H)$ and a Weyl alcove $\AAA$ of $H$ (see \cite{biquard-garcia-prada-mundet}) we thus have that each conjugacy class $\CCC_{x_j}$ corresponds to a weight $\alpha_j \in \sqrt{-1} \AAA$. Since $|\Gamma_{x_j}| = N_j$ then $e^{2 \pi i N_j  \alpha_j} = \id \in H^\C$. In the following we will always choose the weights $\alpha_j$ in the interior of the Weyl alcove $\sqrt{-1} \AAA$.

Given a branch point $y \in Y$ and two points $x, x' \in \pi^{-1}(y)$, there is a deck transformation $\gamma \in \Gamma$ such that $x' = \gamma \cdot x$, and the lift of $\gamma$ to the group of automorphisms of $E$ determines a map on the fibres $\gamma : E_x \rightarrow E_{x'}$. Moreover, the isotropy groups are conjugate $\Gamma_{x'} = \gamma \Gamma_x \gamma^{-1}$ and so the conjugacy classes $\CCC_{x}$ and $\CCC_{x'}$ are equal, and hence so are the weights in $\sqrt{-1} \AAA$ associated to these classes.

Now consider a $\Gamma$-equivariant Higgs structure on $E$, i.e. a holomorphic structure on $E$ together with a Higgs field $\phi$ such that $(E, \phi)$ is preserved by the action of $\Gamma$. For each ramification point $x_j$, choose a small neighbourhood $U_j$ such that the bundle is trivial $\left. E \right|_{U_j} \cong U_j \times H^\C$ and the $\Gamma$-action is trivial 
\begin{equation}\label{eqn:local-trivial-action}
e^{\frac{2\pi i}{N_j}} \cdot (z, g) = (e^{\frac{2\pi i}{N_j}} z, e^{2 \pi i \alpha_j} \cdot g)
\end{equation}
(as explained in \cite{teleman-woodward}, the existence of this trivialisation follows from the equivariant Oka principle of \cite{heinzner-kutschebauch}). We now show that after gauging by $z^{-N_j \alpha_j}$ on each trivialisation for $j = 1, \ldots, r$ then the Higgs pair $(E, \phi)$ descends to a parabolic Higgs bundle on the quotient $(X \setminus \PPP) / \Gamma$, where the weight at the branch point $\pi(x_j)$ is $\alpha_j$. This is known for holomorphic vector bundles (cf. \cite{furuta-steer}, \cite{biswas}) and holomorphic principal bundles (cf. \cite{teleman-woodward}, \cite{balaji-seshadri}), and so to describe the correspondence for Higgs bundles it only remains to describe the residue of the Higgs field at each branch point in $Y$, which is a local computation on each neighbourhood $U_j$. This was worked out for Higgs vector bundles in \cite{nasatyr-steer}, however this has not appeared in the literature for general $G$-Higgs  bundles and so we include the details below.

Locally, the Higgs field on $E$ has the form $\phi(z) = f(z) dz$, where $f(z) : U_j \rightarrow \mathfrak{m}^\C$ is holomorphic. The action of $\Ad_{e^{2 \pi i \alpha_j}}$ decomposes $\mathfrak{m}^\C$ into eigenspaces
\begin{equation*}
\mathfrak{m}^\C = \bigoplus_{\beta} \mathfrak{m}_\beta^\C
\end{equation*}
where $\mathfrak{m}_\beta^\C$ denotes the eigenspace with eigenvalue $e^{2 \pi i \beta}$. Note that each $N_j \beta$ is an integer since $e^{2 \pi i N_j \alpha_j} = \id$, and since $\alpha_j$ is in the interior of the Weyl alcove then each eigenvalue is strictly less than one. Let $f = \bigoplus_\beta f_\beta$ be the corresponding decomposition of $f$. Since each $f_\beta$ is holomorphic then we can write it as a power series
\begin{equation*}
f_\beta(z) = \sum_{k=0}^\infty a_k^\beta z^k
\end{equation*}
with $a_k^\beta$ taking values in $\mathfrak{m}_\beta^\C$. The induced action of $e^{\frac{2\pi i}{N_j}}$ on $\phi$ is given by
\begin{equation*}
e^{\frac{2\pi i}{N_j}} \cdot \phi(z) = \Ad_{e^{2 \pi i \alpha_j}} \left( f \left( e^{\frac{2\pi i}{N_j}} z \right) \right) e^{\frac{2\pi i}{N_j}} dz .
\end{equation*}
Therefore, the action on the component $\phi_\beta = f_\beta dz$ is
\begin{equation*}
e^{\frac{2\pi i}{N_j}} \cdot f_\beta(z) dz = e^{2 \pi i \beta} \sum_{k=0}^\infty a_k^\beta e^{\frac{2 \pi i k}{N_j}} z^k \, e^{\frac{2\pi i}{N_j}} dz = \sum_{k=0}^\infty a_k^\beta e^{\frac{2\pi i(k+1)}{N_j}} e^{2 \pi i \beta} z^k \, dz .
\end{equation*}
If $\phi$ is invariant under the action of $\mathbb{Z}_{N_j} \cong \Gamma_{x_j}$ then we see that $a_k^\beta \neq 0$ implies that $k = N_j \ell - N_j \beta - 1$ for some $\ell \in \mathbb{Z}$.  Therefore
\begin{equation*}
f_\beta(z) dz = \begin{cases} z^{-N_j \beta} \sum_{\ell=0}^\infty a_{N_j \ell-N_j \beta - 1}^\beta z^{N_j \ell} \, z^{-1} dz & \text{if $\beta < 0$} \\ z^{-N_j \beta} \sum_{\ell=1}^\infty a_{N_j \ell-N_j \beta - 1}^\beta z^{N_j \ell} \, z^{-1} dz & \text{if $0 \leq \beta < 1$} \end{cases}
\end{equation*}
where the two distinct cases come from the requirement that $f_\beta$ is holomorphic and hence the power series has non-negative powers of $z$. To simplify the notation, we will use $b_\ell^\beta  = a_{N_j \ell - N_j \beta - 1}^\beta$ in the sequel. On the punctured disk $U_j \setminus \{0\}$, apply the meromorphic gauge transformation $g(z) = z^{N_j \alpha_j} = e^{N_j \alpha_j \log z}$ (note that this is well-defined on the punctured neighbourhood $U_j \setminus \{x_j \}$ since $e^{2\pi i N_j \alpha_j} = \id$). We have $g(z) \cdot \phi(z) = \sum_\beta g(z) \cdot f_\beta(z) dz$ where
\begin{equation*}
g(z) \cdot f_\beta(z) dz = \begin{cases} \sum_{\ell=0}^\infty b_\ell^\beta z^{N_j \ell} \, z^{-1} dz & \text{if $\beta < 0$} \\ \sum_{\ell=0}^\infty b_{\ell+1}^\beta z^{N_j \ell} \, z^{N_j-1} dz & \text{if $0 \leq \beta < 1$} \end{cases}
\end{equation*}
Therefore, after applying the meromorphic gauge transformation $g(z)$, the residue of $g(z) \cdot f_\beta(z)$ is zero if $\beta \geq 0$ and equal to $b_0^\beta$ if $\beta < 0$. Now let $V = \pi(U_j) \subset Y$ and note that \eqref{eqn:local-trivial-action} implies that $\pi : U_j \rightarrow V$ is given by $z \mapsto z^{N_j}$. Then $w = z^{N_j}$ satisfies $w^{-1} dw = N_j z^{-1} dz$ and so $g(z) \cdot f_\beta(z)$ can be written as a function of $w$, i.e. it descends to the quotient $(U_j \setminus \{x_j\}) / \Gamma_{x_j}$
\begin{equation*}
g(z) \cdot f_\beta(z) = f_\beta'(w) = \begin{cases} \sum_{\ell=0}^\infty b_\ell^\beta w^{\ell} \, \frac{1}{N_j} w^{-1} dw & \text{if $\beta < 0$} \\ \sum_{\ell=0}^\infty b_{\ell+1}^\beta w^{\ell} \, \frac{1}{N_j} dw & \text{if $0 \leq \beta < 1$} \end{cases}
\end{equation*}
Therefore the $\Gamma$-invariant Higgs bundle $(E, \phi)$ on $X$ defines a parabolic Higgs bundle $(E', \phi')$ on $Y$ with Higgs field $\varphi' \in \Gamma \left( PE'(\mathfrak{m}^\C) \otimes K(D) \right)$. In particular, the residue of the Higgs field $\varphi'(w) = f'(w) dw$ is $\bigoplus_{\beta < 0} b_0^\beta$ which is nilpotent and so the projection to $\liem_{\alpha_i}^0/L_{\alpha_j}$ is zero. 

Therefore the $\Gamma$-equivariant Higgs bundle $(E, \phi)$ on $X$ with isotropy representations $\sigma$ corresponding to weights $\alpha_j \in \sqrt{-1}\mathcal{A}$ in the interior of the Weyl alcove determines a parabolic Higgs bundle $(E', \phi')$ on $Y$ with parabolic points $\{ y_1, \ldots, y_s \} = \pi (\{ x_1, \ldots, x_r \})$, conjugacy classes $\CCC_{\pi(x_j)}' = \CCC_{x_j}$ determined by $\alpha_j$ and a parabolic Higgs field with nilpotent residues. Moreover, gauge equivalent $\Gamma$-equivariant Higgs bundles on $X$ descend to parabolic gauge-equivalent parabolic Higgs bundles on $Y$.

Conversely, given a parabolic $G$-Higgs bundle $(E', \phi')$ on $Y$ with nilpotent residues at each parabolic point, as above let $V$ be a neighbourhood of a branch point $y$ such that the bundle is trivial over $V \setminus \{y\}$ with weight $\alpha_j \in \sqrt{-1} \mathcal{A}$ such that $e^{2\pi i N_j \alpha_j} = \id$. Since the residues are nilpotent then the Higgs field $\varphi' \in \Gamma \left( PE'(\mathfrak{m}^\C) \otimes K(D) \right)$ has the form
\begin{equation*}
f_\beta'(w) = \begin{cases} \sum_{\ell=0}^\infty c_{\ell}^\beta w^{\ell} \, w^{-1} dw & \text{if $\beta < 0$} \\ \sum_{\ell=0}^\infty c_{\ell}^\beta w^{\ell} \,  dw & \text{if $0 \leq \beta < 1$} \end{cases}
\end{equation*}
After pulling back by the ramified covering map $z \mapsto z^{N_j} = w$, the Higgs field $\phi(z) = f(z) dz$ upstairs has the form
\begin{equation*}
f_\beta(z) = \begin{cases} \sum_{\ell=0}^\infty c_{\ell}^\beta z^{N_j \ell} \, N_j z^{-1} dz & \text{if $\beta < 0$} \\ \sum_{\ell=0}^\infty c_{\ell}^\beta z^{N_j \ell} \, N_j z^{N_j-1} dz & \text{if $0 \leq \beta < 1$} \end{cases}
\end{equation*}
Applying the gauge transformation $g(z) = z^{-N_j \alpha_j}$ (once again, $e^{2\pi i N_j \alpha_j} = \id$ implies that this is well-defined on the punctured neighbourhood $U_j \setminus \{x_j\}$) gives us
\begin{equation*}
g(z) \cdot f_\beta(z) = \begin{cases} z^{-N_j \beta} \sum_{\ell=0}^\infty c_{\ell}^\beta z^{N_j \ell} \, N_j z^{-1} dz & \text{if $\beta < 0$} \\ z^{-N_j \beta} \sum_{\ell=0}^\infty c_{\ell}^\beta z^{N_j \ell} \, N_j z^{N_j-1} dz & \text{if $0 \leq \beta < 1$} \end{cases}
\end{equation*}
and the same argument as before shows that this is holomorphic and invariant under the action of $\mathbb{Z}_{N_j}$ determined by $\alpha_j \in \sqrt{-1} \mathcal{A}$. Therefore the parabolic Higgs bundle on $Y$ determines a $\Gamma$-equivariant Higgs bundle on $X$.

Now that we have established the correspondence, it only remains to show that the notions of stability,  semistability and polystability 
 are also in correspondence. In the case of holomorphic principal bundles, the results of \cite[Sec. 2.2]{teleman-woodward} show that, via the correspondence described above, a stable $\Gamma$-equivariant bundle upstairs on $X$ corresponds to a stable parabolic bundle on $Y$. 
Moreover, the degree of any parabolic reduction of structure group on $E \rightarrow X$ is related to the parabolic degree of a parabolic reduction of structure group on $E' \rightarrow Y$ by a factor of $\frac{1}{|\Gamma|}$.

For Higgs bundles, the only modification is to restrict to reductions of structure group which are compatible with the Higgs field as described in \cite[Sec. 3.2]{biquard-garcia-prada-mundet}. For the Higgs bundle $(E, \phi)$ over $X$, given $s \in \sqrt{-1}\mathfrak{h}$ and a $\Gamma$-invariant holomorphic reduction $\eta \in \Omega^0(E(H^\C / P_s))$ such that $\varphi \in H^0(X, E_\eta(\mathfrak{m}_s) \otimes K)$, the $\Gamma$-invariance of the Higgs field $\phi$ implies that the induced reduction of structure group on the parabolic bundle $(E', \phi')$ over $Y \setminus \SSS$ is compatible with the Higgs field, i.e. $\left. \phi' \right|_{Y \setminus \SSS} \in H^0(Y \setminus \SSS, E_\eta'(\mathfrak{m}_s) \otimes K)$. Conversely, a reduction of structure group on the parabolic bundle $(E', \phi')$ over $Y \setminus \SSS$ which is compatible with the Higgs field $\phi'$ lifts to a reduction of $(E, \phi)$ over $X$ compatible with $\phi$. Since the degree on $X$ is related to the parabolic degree on $Y$ by a factor of $\frac{1}{|\Gamma|}$ (cf. \cite[Sec. 2.3]{teleman-woodward}) then the notion of $\Gamma$-equivariant Higgs stability (resp. semistability and polystability) upstairs on $X$ corresponds to the notion of parabolic Higgs stability (resp. semistability and polystability) downstairs on $Y$.


Therefore we have proved the following bijection of moduli spaces.

\begin{theorem}\label{equivariant-parabolic}

The correspondence described above defines a bijection
$$
\cM(X,G,\Gamma, \id, \sigma) \rightarrow  \cM(Y,\SSS, G,\alpha, 0).
$$
\end{theorem}

\subsection{$\tau$-Twisted $\Gamma$-equivariant structures and 
pseudoreal  parabolic Higgs bundles}\label{pseudo-real-parabolic}

In this section we assume that $\Gamma$ contains antiholomorphic automorphisms 
$\Gamma$ contains antiholomorphic automorphisms 
of $X$, that is,  $\Gamma$ is given by an extension
$$
1\to \Gamma^+\to \Gamma \to \Z/2\to 1
$$
defined by (\ref{modular-extension}).
We also assume that the $(\Gamma,\tau,c)$-equivariant structures on
the $G$-Higgs bundles over $X$ are  such that the restriction of the cocycle 
$c$  to $\Gamma^+$ is trivial.  In this situation $c$ defines a cocycle 
$\widetilde{c}\in Z^2_\tau(\Z/2, Z')$ where the action of 
$\Z/2=\Gamma/\Gamma^+$ is the one  induced by (\ref{action-on-Z}).
The cocycle $\widetilde{c}$ defines a pseudoreal structure on the parabolic 
$G$-Higgs bundle on $Y:=X/\Gamma^+$ constructed in the previous section.
Pseudoreal structures of parabolic $G$-Higgs bundles are studied in \cite{calvo-garcia-prada-perez},
generalising the theory of pseudoreal $G$-Higgs bundles is well-understood 
\cite{biswas-garcia-prada-hurtubise,biswas-garcia-prada,biswas-calvo-garcia-prada}.
One thus has a correspondence similar to the one in Theorem \ref{equivariant-parabolic}
also in this situation.




\section{Twisted equivariant structures and representations}

In this section, $S$ is an oriented smooth compact surface of genus $g\geq 2$,
$X$ is a Riemann surface, whose  underlying smooth surface is  $S$. The Lie 
group $G$ is a real form of a complex semisimple Lie group $G^\C$, 
and $\tau$ is a 
conjugation of $G^\C$ defining a compact real form of $G^\C$, and preserving $G$. Finally, $\Gamma$ is a subgroup of $\Aut(X)$.

\subsection{Twisted equivariant Higgs bundles and the orbifold fundamental group}

Exploiting Proposition \ref{action-correspondence} and Theorem 
\ref{fixed-points-theorem}, we will give an interpretation of 
 the fix-point locus $\calR(S,G)^\Gamma$ in terms of representations of the $\Gamma$-orbifold 
fundamental group $\pi_1(S,\Gamma)$
of $S$ (see \cite{biswas-garcia-prada-hurtubise}, for example, for a 
definition). This group fits into a short exact sequence 
$$
1\to \pi_1(S) \to \pi_1(S,\Gamma)\to  \Gamma\to 1.
$$

Let $c\in Z^2_\tau(\Gamma,Z)$ be a $2$-cocycle, where $Z$ is the centre of $G$.
Recall that $\Gamma$ acts on $G$ as
$ g^\gamma=g$ if $\gamma\in \Gamma^+$  and  
$g^\gamma=\tau(g)$ if $\gamma\in \Gamma^-$, for $g\in G$, 
inducing  an action on $Z$. Let also 
$$
 \tau^\gamma= \begin{cases}
\Id&\mbox{if}\;\;  \gamma\in \Gamma^+\\
\tau &\mbox{if}\;\; \gamma\in \Gamma^-.
\end{cases}
$$

We consider a group $\widehat{G}=\widehat{G}(\Gamma,\tau,c)$,
whose set is $G\times \Gamma$, and the  group structure is defined by
$$
(g_1,\gamma_1)\cdot (g_2,\gamma_2)= (g_1\tau^{\gamma_1}(g_2)c(\gamma_1,\gamma_2),\gamma_1\gamma_2),
$$
for $g_1,g_2\in G$ and $\gamma_1,\gamma_2\in\Gamma$.

Let  $\calR(S,G,\Gamma,\tau,c)$ be the set of $G$-conjugacy classes  of 
 representations 
$\widehat{\rho}:\pi_1(S,\Gamma)\to \widehat{G}$, which extend reductive   
representations $\rho:\pi_1(S)\to G$,  making the following diagram commutative

\begin{displaymath}
  \begin{CD}
  0@>>>\pi_1(S) @>>>\pi_1(S,\Gamma)@>>> \Gamma@>>>1\\
  @.@V \rho VV@V \widehat{\rho} VV@V \Id VV\\
  1@>>> G @>>>\widehat{G}@>>>\Gamma@>>>1.
  \end{CD}
  \end{displaymath} 

We will denote by  $\widetilde{\calR}(S,G,\Gamma,\tau,c)$ the image of  $\calR(S,G,\Gamma,\tau,c)$ in $\calR(S,G)$, 
for the map given by sending $\widehat{\rho}$ to $\rho$. 

Following the arguments in \cite{biswas-garcia-prada-hurtubise} (see also \cite{biswas-garcia-prada,biswas-calvo-garcia-prada}),
we can extend the non-abelian Hodge correspondence to the twisted equivariant case and prove the following.

\begin{theorem}\label{equivariant-nahc}
Under the non-abelian Hodge correspondence given by Theorem  \ref{na-Hodge} one has the homeomorphism
$$
\calR(S,G,\Gamma,\tau,c) \stackrel{\cong}{\longrightarrow} \cM(X,G,\Gamma,\tau,c).
$$
\end{theorem}

Fixing the elements $\sigma_i$ at the points $x_i\in \PPP$ to define the moduli space   $\cM(X,G,\Gamma,\tau,c,\sigma)$ with 
$\sigma=(\sigma_1,\cdots,\sigma_r)$, corresponds to fixing a cyclic element for 
the image under the representation of a loop around the point $x_i$. The moduli space corresponding to $\cM(X,G,\Gamma,\tau,c,\sigma)$
will be denoted by $\calR(S,G,\Gamma,\tau,c,\sigma)$,
and its image in $\calR(S,G)$ by 
$\widetilde{\calR}(S,G,\Gamma,\tau,c,\sigma)$.

As a corollary of Theorems \ref{fixed-points-theorem}  and 
\ref{equivariant-nahc} we have the following.

\begin{theorem}\label{fixed-points-theorem-rep}
(1) For  any cocycle $c\in Z^2_\tau(\Gamma,Z)$
$$
\widetilde{\calR}(S,G,\Gamma, \tau,c)\subset \calR(S,G)^\Gamma.
$$

(2) Let $\calR_*(S,G)\subset \calR(S,G)$ be  the subvariety of irreducible 
representations, then 
$$
\calR_*(S,G)^\Gamma\subset \bigcup_{[c]\in H^2_\tau(\Gamma,Z)} 
\widetilde{\calR}(S,G,\Gamma,\tau,c).
$$
\end{theorem}
\subsection{The orbifold fundamental group and punctured surfaces}

Assume now, as in Section \ref{sec:gamma-hol}, that $\Gamma=\Gamma^+$ and that the cocycle $c$ is trivial. 
Let $\SSS$ be  the set of points in $S/\Gamma$ corresponding to the points $\PPP\subset S$  (see Section \ref{sec:gamma-hol}).
Then,  combining Theorems \ref{equivariant-parabolic} and \ref{equivariant-nahc}  with the non-abelian Hodge correspondence 
for punctured surfaces,  proved in \cite{biquard-garcia-prada-mundet}, we have the following.

\begin{theorem}\label{equiv-rep-punctures}
There is a bijection between $\calR(S,G,\Gamma,\tau,c=1,\sigma)$ 
and $\calR(S/\Gamma\setminus \SSS,G)$ with conjugacy classes 
around the points in $\SSS$ determined by $\sigma$.
\end{theorem} 

We now assume, as in Section \ref{pseudo-real-parabolic},  that $\Gamma$ contains antiholomorphic automorphisms
and  that the restriction of the cocycle 
$c$  to $\Gamma^+$ is trivial. As mentioned in  Section \ref{pseudo-real-parabolic},  in this situation $c$ defines a cocycle 
$\widetilde{c}\in Z^2_\tau(\Z/2, Z)$ where the action of 
$\Z/2=\Gamma/\Gamma^+$ is the one  induced by (\ref{action-on-Z}). We can then consider the $\Z/2$-orbifold fundamental group
$\pi_1(S/\Gamma^+\setminus \SSS,\Z/2)$ for the residual action of $\Z/2=\Gamma/\Gamma^+$ on $S/\Gamma^+$,
which fits in a short exact sequence 
$$
1\to \pi_1(S/\Gamma^+\setminus\SSS) \to  \pi_1(S/\Gamma^+\setminus\SSS,\Z/2)  \to  \Z/2\to 1.
$$
Here $\SSS$ is the set of points $S/\Gamma^+$ corresponding to the set 
$\PPP\subset S$.

We define the group $\widehat{G}=\widehat{G}(\tau,c)$,
whose set is $G\times \Z/2$, and the  group structure is defined by
$$
(g_1,e_1)\cdot (g_2,e_2)= (g_1\tau^{e_1}(g_2)\widetilde{c}(e_1,e_2),e_1e_2),
$$
for $g_1,g_2\in G$ and $e_1,e_2\in\Z/2$. Here $\tau^{e_1}=\Id$ if $e_1=1$ and  $\tau^{e_1}=\tau$ if 
$e_1=-1$.
Define 
$$
\calR(S/\Gamma^+\setminus \SSS,G,\tau,c)
$$ 
as the set $G$-conjugacy classes of homomorphisms
 $\widehat{\rho}:\pi_1(S/\Gamma^+\setminus \SSS,\Z/2)\to \widehat{G}$ 
extending homomorphisms $\rho:\pi_1(S/\Gamma^+\setminus \SSS)\to G$, with
$\rho$ reductive, and making the 
the following diagram commutative

\begin{displaymath}
  \begin{CD}
  0@>>>\pi_1(S/\Gamma^+\setminus \SSS) @>>>\pi_1(S/\Gamma^+\setminus \SSS,\Z/2)@>>> \Z/2@>>>1\\
  @.@V \rho VV@V \widehat{\rho} VV@V \Id VV\\
  1@>>> G @>>>\widehat{G}@>>>\Z/2@>>>1.
  \end{CD}
  \end{displaymath} 
From  the discussion in Section  \ref{pseudo-real-parabolic} we conclude the following.

\begin{theorem}\label{equiv-rep-punctures-pseudo}
There is a bijection between the moduli spaces  $\calR(S,G,\Gamma,\tau,c,\sigma)$ and $\calR(S/\Gamma^+\setminus \SSS,G,\tau,c)$ with conjugacy classes 
around the points in $\SSS$ determined by $\sigma$.
\end{theorem}

\section*{Acknowledgements}

The first author was partially supported by the Spanish MINECO under the ICMAT Severo Ochoa 
grant No.SEV-2015-0554, and grant No. MTM2013-43963-P and by the European
Commission Marie Curie IRSES  MODULI Programme PIRSES-GA-2013-612534.  The second author was partially supported by Singapore Ministry of Education Academic Research Fund
Tier 1 grant number R-146-000-200-112.

\providecommand{\bysame}{\leavevmode\hbox to3em{\hrulefill}\thinspace}

\end{document}